\documentclass[11pt,reqno,a4paper]{amsart}
\usepackage[toc,page]{appendix}

\usepackage[utf8]{inputenc}
\usepackage{esint}
\usepackage{MnSymbol}
\usepackage{enumitem}
\usepackage[english]{babel}
\usepackage{amsmath}
\usepackage{thm-restate}

\usepackage[
	colorlinks,
	pdfpagelabels,
	pdfstartview = FitH,
	bookmarksopen = true,
	bookmarksnumbered = true,
	linkcolor = blue,
	plainpages = false,
	hypertexnames = false,
	citecolor = red] {hyperref}

\hypersetup{
	linktoc=page
}

\usepackage[capitalize]{cleveref}
\crefname{enumi}{}{}
\crefname{equation}{}{}

\makeatletter
\def\@tocline#1#2#3#4#5#6#7{\relax
  \ifnum #1>\c@tocdepth 
  \else
    \par \addpenalty\@secpenalty\addvspace{#2}%
    \begingroup \hyphenpenalty\@M
    \@ifempty{#4}{%
      \@tempdima\csname r@tocindent\number#1\endcsname\relax
    }{%
      \@tempdima#4\relax
    }%
    \parindent\z@ \leftskip#3\relax \advance\leftskip\@tempdima\relax
    \rightskip\@pnumwidth plus4em \parfillskip-\@pnumwidth
    #5\leavevmode\hskip-\@tempdima
      \ifcase #1
       \or\or \hskip 1em \or \hskip 2em \else \hskip 3em \fi%
      #6\nobreak\relax
    \dotfill\hbox to\@pnumwidth{\@tocpagenum{#7}}\par
    \nobreak
    \endgroup
  \fi}
\makeatother

\newtheorem{theorem}{Theorem}
\newtheorem*{theorem*}{Theorem}
\newtheorem{proposition}{Proposition}[section]
\newtheorem{lemma}[proposition]{Lemma}

\theoremstyle{definition}
\newtheorem{definition}[proposition]{Definition}
\newtheorem{remark}[proposition]{Remark}
\newtheorem{example}[proposition]{Example}
\numberwithin{equation}{section}

\def \R {\mathbb {R}}
\def \N {\mathbb {N}}
\def \E {\mathbb{E}}

\def\osc{\operatorname{osc}}

\def\grad{\nabla}

\renewcommand{\tilde}{\widetilde}

\renewcommand{\S}{\mathbf{S}}

\newcommand{\Hma}{\mathcal{H}_{m,a}}
\newcommand{\cHma}{\overline{\mathcal{H}}_{m,a}}
\newcommand{\Lal}{\mathcal{L}_{al}}

\newcommand{\I}{\mathbb{I}}
\newcommand{\eps}{\varepsilon}

\newcommand{\pa}{\partial}

\renewcommand{\fint}{\strokedint}
\newcommand{\sgn}{\operatorname{sgn}}

\allowdisplaybreaks

\begin{document}

\title[Approximation of BV functions by neural networks]{Approximation of BV functions by neural networks\\ {\tiny A regularity theory approach}}
\author[Avelin]{Benny Avelin}
\address{Benny Avelin,
Department of Mathematics, 
Uppsala University,
S-751 06 Uppsala, 
Sweden} 
\email{\color{blue} benny.avelin@math.uu.se}
\author[Julin]{Vesa Julin}
\address{Vesa Julin,
Department of Mathematics and Statistics, 
University of Jyv\"askyl\"a,
P.O. Box 35, 
40014 Jyv\"askyl\"a, 
Finland} 
\email{\color{blue} vesa.julin@jyu.fi}


\date{\today}

\begin{abstract}
    In this paper we are concerned with the approximation of functions by single hidden layer neural networks with ReLU activation functions on the unit circle. In particular, we are interested in the case  when the number of data-points exceeds the number of nodes. We first study the convergence to equilibrium of the  stochastic gradient flow associated with the cost function with a quadratic penalization. Specifically, we prove a Poincar\'e inequality for a penalized version of the cost function with explicit constants that are independent of the data and of the number of nodes. As our penalization biases the weights to be bounded, this leads us to study how well a network with bounded weights can approximate a given function of bounded variation (BV). 
    
    Our main contribution concerning approximation of BV functions, is a result which we call \emph{the localization theorem}. Specifically, it states that the expected error of the constrained problem, where the length of the weights are less than $R$, is of order $R^{-1/9}$ with respect to the unconstrained problem (the global optimum). The proof is novel in this topic and is inspired by techniques from regularity theory of elliptic partial differential equations. Finally we quantify the value of the global optimum by proving a quantitative version of the universal approximation theorem.
\end{abstract}

\maketitle
\tableofcontents

\section{Introduction}

This paper is concerned with a supervised learning problem, where the goal is to learn a function from data. We adopt the viewpoint of function approximation. Specifically, we consider $x \in \S^1$, and a function of bounded variation, or BV-function for short, $y: \S^1 \to \R$, where $\S^1$ is the unit circle in $\R^2$. The goal is to approximate the function $y$ in  the space of two layer neural networks (single hidden layer) with ReLU activation on the unit circle. To be specific let us define our space of functions as
\begin{equation}\label{def f}
	f_{W,a}(x) = \frac{1}{\sqrt{m}}\sum_{i=1}^m a_i \sigma(w_i \cdot x),
\end{equation}
where the vectors $w_i \in \R^2$, $W = (w_1, \dots, w_m)\in \R^{2m}$, denote the weights, and the coefficients $a_i \in \{-1,1\}$,  $a = (a_1, \dots, a_m)$, are given and the activation function is  $\sigma(t) = \max\{t,0 \}$. We denote the set of functions formed in \cref{def f} as $\Hma$.
This particular set of neural networks on the unit circle (or unit sphere) has previously been studied e.g. in  \cite{Arora, DuZhai}. Note that the network in \cref{def f} does not have a bias, which is a simplifying assumption. However, it turns out that this space is still rich enough to approximate any antipodally symmetric function. We assume that the function $y$ is in BV, which is a natural assumption as this allows it to take discrete values which is common in classification problems. Moreover, since the Fourier coefficients $a_k$ of a periodic BV-function decays as $k^{-1}$, this class contains a periodic Barron type space, \cite{Barron}.

We define the cost function $\Phi : \R^{2m} \to \R$ as
\begin{equation} \label{def:costfunc}
\Phi(W) :=  \int_{\S^1} |f_{W,a}(x) - y(x)|^2 d\mu(x) = \|f_{W,a} - y\|_{L^2(\mu)}^2
\end{equation}
and consider the following $L^2$ approximation problem
\begin{equation} \label{min prob}
	\inf_W \Phi(W).
\end{equation}
In \cref{def:costfunc}  $\mu$ is the data-specific measure, which we assume to be a probability measure on $\S^1$. In particular, $\mu$ can be discrete or  absolutely continuous with respect to the Lebesgue measure. This covers all the interesting cases as we have no restriction on $\mu$.
    
The paper is divided into two parts. We first consider first-order methods with noise (Langevin dynamics) to find the minimum of \cref{min prob}.
It turns out that for a general probability measure $\mu$ we need to add a penalization on the cost function \cref{def:costfunc} in order to have a uniform convergence rate (see \cref{badexamp}). This is in contrast to \cite{DuZhai} where the authors assume a non-degeneracy condition on the empirical measure, i.e. that the points of support are mutually non-parallel. It is clear that when the number of data points grows to infinity we will lose the non-degeneracy condition. In this direction and for a general probability  measure $\mu$, we prove in \cref{thm1} the exponential convergence of the stochastic gradient flow associated with \cref{min prob} under quadratic penalization. 

In the second part we study the value of \cref{min prob}, with and without penalization, which we interpret as function approximation problems. As our penalization biases the weights to have bounded length, we first estimate in \cref{thm2} the expected error that we make when we approximate $y$ under the  constraint that the functions in our network $\Hma$  have uniformly bounded weights. We call this \emph{the Localization theorem} and it is the main result of this paper. To the best of our knowledge the closest related result in the literature is \cite{bolcskei2019optimal}. Finally in \cref{thm3} we prove a uniform approximation theorem where we estimate the value of \cref{min prob} in terms of number of (alternating) nodes. This is a quantitative form of the universal approximation theorem \cite{Barron}, see also \cite{shaham2018provable}.

\subsection{Statement of main results}
\subsubsection*{Convergence of stochastic first order methods}
\phantom{.}

Consider the following stochastic optimization of the minimization problem \cref{min prob}:
\begin{align*}
	dW_t = - \grad \Phi(W) dt + \sqrt{2} \eps dB_t,
\end{align*}
where $dB_t$ is the standard isotropic Wiener noise. We remark that the above SDE is basic in many areas of physics and can be considered as a weak approximation to Stochastic Gradient Descent, as in \cite{AN,LTE,LTW,SSJ}. The density $\rho(W,t)$ of $W_t$ satisfies the following partial differential equation (PDE), called the Fokker-Planck or forward Kolmogorov equation,
\begin{align*}
	\frac{\partial \rho}{\partial t} = \nabla \cdot \left ( \rho \nabla \Phi + \eps^2 \nabla \rho \right ).
\end{align*}
The smoothness imposed by the isotropic noise term is crucial and it seems the way forward when dealing with general data-measures $\mu$. Naturally the noise with $\eps>0$ also helps the flow not to get stuck near local minimum points.

The convergence rate of the density $\rho(\cdot, t)$ to an equilibrium measure $\rho_\infty$ and its relation to  Poincar\'e inequalities (or equivalently spectral gaps) and logarithmic Sobolev inequalities is a well studied problem in both stochastics and in PDEs, see e.g. \cite{Aida,Aida2,BBCG,Bobkov,Cattiaux,CG,DPF,Gross,Milman,VBook,Wang}.  In particular, if we want to prove an exponential convergence rate for any initial distribution $\rho_0$,  $\Phi$ needs to be a confining potential, i.e. it needs to grow at infinity with at least linear rate \cite{VBook}. This is the motivation for studying $\Phi$ in depth.

Let us first discuss about the minimization problem \cref{min prob} itself, and, in particular, why it is in general not well posed. By this we mean that the problem does not always admit a minimizer and the gradient flow may not converge.  The reason for this is that $\Phi$ is not coercive, which causes problems as the following example shows.  
\begin{example} \label{badexamp}
    We consider $\mathcal{H}_{2,a}$, which is the simplest non-trivial neural network, and  assume that $\mu$ is the normalized Lebesgue measure. Let $a_1 = 1$ and $a_2 = -1$ and choose
	\begin{align*}
		y(x) = y(x_1,x_2) = \I\{x_2 \geq 0\} x_1,
	\end{align*}
	where $\I$ denotes the indicator function. Note that $y$ is not continuous but it is a function of bounded variation. Then $\Phi(W) > 0$ for all $W \in  \R^4$ but
	\begin{align*}
		\inf_{W \in \R^4} \Phi(W) = 0.
	\end{align*}
	In particular, the function $\Phi$ does not have a global minimum. Moreover,  there is a point  $W_0  \in \R^4$ such that  the gradient flow 
    \begin{align*}
        \frac{d}{dt} W_t = - \nabla \Phi(W_t),
    \end{align*}
    staring from  $W_0$ diverges, i.e.,
    \begin{align*}
        \lim_{t \to \infty} |W_t| = \infty.
    \end{align*}
\end{example}

Let us briefly explain why the example holds. First, since $y$ is not continuous and every $f_{W,a} \in \Hma$ is Lipschitz continuous, it is clear that for every $W$ we have $\Phi(W) >0$. To see that the infimum is zero we notice that the following holds:
\begin{align*}
	\lim_{h \to 0} \frac{\sigma(x_2 + h x_1) - \sigma(x_2) }{h} =  \sigma'(x_2) x_1 = \I\{x_2 \geq 0\} x_1,
\end{align*}
pointwise a.e. $x$. Therefore, by recalling that $f_{W,a}(x) = \frac{1}{\sqrt{2}} (\sigma(w_1 \cdot x) - \sigma(w_2 \cdot x))$ we have by choosing $w_1 = \frac{\sqrt{2}}{h} e_2 + \sqrt{2} e_1$ and $w_2 = \frac{\sqrt{2}}{h} e_2$ that
\begin{align*}
	f_{W,a}(x) =  \sigma( \frac{1}{h} x_2+ x_1) - \sigma(\frac{1}{h} x_2 ) = \frac{\sigma(x_2 + x_1) - \sigma(x_2) }{h}  \to \I\{x_2 \geq 0\} x_1.
\end{align*}
The claim $\inf \Phi(W) = 0$ then follows by the dominated convergence theorem. This shows that the minimization problem \cref{min prob} does not, in general, have a minimum. \cref{badexamp} also claims that the gradient flow does not in general converge. We prove this in \cref{lem:badexamp2} in \cref{sec:thm1}. Furthermore, if the target function is a neural network in the class $\Hma$, then the global minimum is only unique up to certain symmetries, see for instance \cite{fefferman1993recovering,fefferman1994reconstructing,vlavcic2019neural}.

\cref{badexamp} shows that for a generic probability measure $\mu$ the minimization problem \cref{min prob} is surprisingly complicated already in the two node case. We do not try to classify all critical points of \cref{min prob} but \cref{badexamp} shows that the problem is not coercive, does not have global minimum point, and the gradient flow may diverge. This suggests that without any assumptions on the data-points, the exponential convergence proved in \cite{DuZhai} probably fails. 

As mentioned before, in order for the density of the Fokker-Planck equation to converge exponentially fast to equilibrium, the potential needs to be confining, but \cref{badexamp} shows that this is not the case. Motivated by this, we modify the cost function \cref{min prob} by adding a penalization, similar to \cite{AN}, and define for a chosen large parameter $R$
\begin{equation}\label{pena}
	\Phi_R(W) :=  \max\{ \Phi(W), 4(|W|^2 - R^2) \}.
\end{equation}
Now it is trivial to see that for every $R>0$, the problem
\begin{equation}\label{min prob delta}
	\inf_{W \in \R^{2m}} \Phi_R(W)
\end{equation}
has a global minimum, and at a minimum point $W_R$ the weight is bounded by $|W_R|\leq R+ \|y\|_{L^2(\mu)}$.

In the  first result  we study the convergence of the stochastic gradient flow of the penalized cost function
\begin{equation} \label{SDE2}
	dW_t = - \grad \Phi_R(W) dt + \sqrt{2} \eps dB_t
\end{equation}
with fixed small $0< \eps\leq 1$. Again we denote the density of $W_t$ by $\rho(W,t)$.

\begin{theorem}
	\label{thm1}
	Let us fix $\eps \in (0,1]$ and   choose $R\geq 10$. Let $\Phi_R$ be the penalized cost function defined by  \cref{def:costfunc,pena}, and assume $\|y\|_{L^2(\mu)} \leq 1$. Assume that $\rho(W,t)$ is the density of $W_t$ which is the solution of  \cref{SDE2}   with initial datum $\rho_0$. Then $\rho(\cdot, t) \to \rho_\infty$, where $\rho_\infty(W) = e^{- \Phi_R(W)/ \eps^2}$,  exponentially fast and we have the estimate
	\begin{align*}
		\int_{\R^{2m}} |\rho(\cdot, t) - \rho_\infty|^2  e^{ \Phi_R/ \eps^2} \, dW  \leq e^{-C t} \int_{\R^{2m}} |\rho_0 - \rho_\infty|^2  e^{ \Phi_R / \eps^2} \, dW.
	\end{align*}
	When the number of nodes satisfies   $m \geq 24 R^2 \eps^{-2}$, the constant $C$ is bounded by 
	\begin{align*}
		C \geq \frac{1}{70}.
	\end{align*}
	Otherwise we have the bound $C \geq 16 e^{-(4R^2 +2)\eps^{-2}}$.
	
\end{theorem}

The strength of \cref{thm1} is that we obtain the exponential convergence of the stochastic gradient flow \cref{SDE2} and the rate is independent of the number of nodes and on the distribution of the data points. The latter follows from the fact that the rate does not depend on the probability measure $\mu$. In addition, \cref{thm1} implies that when the number of nodes is large enough with respect to $\eps^{-1}$ and  $R$, then the stochastic gradient flow \cref{SDE2} converges with a uniform speed (in distribution). The convergence rate in \cref{thm1} does however depend on $\eps$ and on $R$, but this is necessary. Indeed, \cref{badexamp} indicates that the dependence on $R$ is necessary, since otherwise it is not clear if the flow even  converges. 

We prove \cref{thm1} using Aida's perturbation argument \cite{Aida}. Specifically, we consider $\Phi_R$ as a perturbation of a quadratic potential which satisfies a dimensionally independent log-Sobolev inequality.
Furthermore, when the dimension is high, $\Phi_R$ can be seen as a local perturbation of the quadratic potential. For this reason we are able to obtain a dimension-free bound on the spectral gap when  $m \geq R^2\eps^{-2}$. We note that it would be interesting to consider other penalizations, e.g.  the linear growth penalization $c|W|$ or the quadratic penalization $c |W|^2$. Unfortunately, in both of these cases the perturbation proof breaks down when $c$ is small. For linear growth potentials the dimensionally independent result is still lacking, and the best known result  can be found in \cite{BBCG,BCG,Wu}.

\subsubsection*{Function approximation}
\phantom{.}

\cref{thm1} provides a uniform rate of convergence to equilibrium of \cref{SDE2}. Furthermore, note that when $\eps>0$ is small the equilibrium density $\rho_\infty$ is concentrated near the global minimum of the cost function $\Phi_R$ defined in \cref{pena}. However, this still leaves open the question, what is the minimum value \cref{min prob delta} and how close it is to \cref{min prob}?  We restate this question in an equivalent way as a problem of function approximation, i.e., how well can we approximate a given function $y$ using functions in our space $\Hma$. Note that it is clear from our choice of the penalization \cref{pena} that 
\[
    \min_{ |W|\leq R} \Phi(W) \geq \inf_{W} \Phi_R(W) \geq \inf_{W} \Phi(W),
\]
which leads to a further question, how well we can approximate $y$ using functions from $\Hma$ under the constraint that the weights are bounded by $|W|\leq R$?

To this aim, in the second part of the paper we study the values of
\begin{align}
    \label{eq:l2approx}
    \inf_{f_W \in \Hma} \| f_W - y \|_{2}^2
\end{align}
and
\begin{align}
    \label{eq:l2_loc_approx}
    \min_{\substack{f_W \in \Hma \\ |W| \leq R}}\| f_W - y \|_{2}^2,
\end{align}
where $f_W = f_{W,a} \in \Hma$ is a function in our network defined in \cref{def f} and $\|\cdot \|_2$ is the standard $L^2$-norm on $\S^1$.  
We are interested in the following questions:
\begin{enumerate}
    \item \label{q:1} How large is the gap between \cref{eq:l2_loc_approx} and \cref{eq:l2approx}?
    \item \label{q:2} What is the value of \cref{eq:l2approx}?
\end{enumerate}
\begin{remark} 
    Note that we measure the distance between $f_W$ and $y$ with the standard $L^2$-norm, which can be interpreted as the expected error between $f(X)$ and $y(X)$ when the random variable  $X$ is uniformly distributed on  $\S^1$. In particular, if we let $\mu$ be the uniform measure, then $\Phi(W) = \| f_W - y \|_{2}^2$.  It would be interesting to answer the Questions \cref{q:1,q:2} for more general probability measures (see \cref{restr:lebesgue})
\end{remark}

The difference between the approximation problems \cref{eq:l2approx} and \cref{eq:l2_loc_approx} is that under the constraint $|W|\leq R$ every function $f \in \Hma$ is Lipschitz continuous with Lipschitz constant $R$. Since we assume $y \in BV(\S^1)$, then in \cref{eq:l2_loc_approx} we are trying to approximate a given BV-function with uniformly Lipschitz continuous functions given by our network. Therefore, we see \cref{eq:l2_loc_approx} as a regularized version of \cref{eq:l2approx}. That is, to answer Question \cref{q:1}, is to study, what is the error that we make when we regularize the problem \cref{eq:l2approx} by the constraint $|W| \leq R$ for $f_W \in \Hma$. We call this \emph{the localization problem} and prove it in \cref{thm2}. 

On the other hand, in \cref{eq:l2approx}, we are trying to approximate BV-functions with functions in our network, but this infimum is not in general attained, see \cref{badexamp}. Thus we first characterize the closure of the space $\Hma$ in \cref{thm4} in \cref{sec:closure} and denote it by $\cHma$. This turns out to be  a subset of BV-functions. To answer Question \cref{q:2}, we prove in \cref{thm3} that the space $\cHma$ asymptotically converges (w.r.t. the number of alternating nodes) to the space of certain BV-functions and quantify this convergence. 

\subsubsection*{The Localization theorem}
\phantom{.}

Let us focus on answering Question \cref{q:1}. Our aim is to estimate the gap between \cref{eq:l2_loc_approx} and \cref{eq:l2approx} in a quantitative way such that the bound depends only on the $BV$-norm $y$. In particular, we want the estimate to be independent of number of nodes $m$. This turns out to be a difficult and much deeper problem than estimating the size of \cref{eq:l2approx}. 

In order to state the result we define the index sets $I = \{  1, \dots, m \}$, $I_+ := \{ i\in I : a_i = 1\}$ and $I_- := \{ i \in I : a_i = -1\}$. We define further   
\begin{align} \label{def:underbar}
    \underbar m := \min \{  \# I_+ ,   \# I_-  \},
\end{align}
and  
\begin{align} \label{def:c_m}
    C(m) := \sqrt{m/\underbar m} \geq \sqrt{2},
\end{align}
with the convention that $C(m) = \infty$ if $\underbar m = 0$. The constant $C(m)$ measures the ratio between the number of positive and negative coefficients $(a_i)_{i \in I}$ and we do not make any a priori assumptions on this value. Our second main result answers Question \cref{q:1} and reads as follows.  

\begin{restatable}[Localization theorem]{theorem}{thmlocalization}
 	\label{thm2}
 	Assume that $y \in BV(\S^1)$ is such that $\|y\|_{L^2(\S^1)} \leq 1$.  Then for all $R \geq R_0$ the following holds
 	\begin{align*}
 		\min_{\substack{f_W \in \Hma \\ |W| \leq C(m)R}}&\| f_W - y \|_{L^2(\S^1)}^2 \\
 		&\leq  \inf_{f_W \in \Hma}\| f_W - y \|_{L^2(\S^1)}^2 + 5 \cdot 10^4 (\|y\|_{BV}^{2}+1)\frac{1}{R^{1/9}},
 	\end{align*}
    where $C(m)$ is defined in \cref{def:c_m} and 
    \[
        R_0 =  \max \{(10 \|y\|_{BV})^6,  4 \cdot 10^7  \}.
    \]
\end{restatable}
Let us make a few comments on this result. First, \cref{thm2} shows that when $R$ is large with respect to the BV-norm of $y$,  the value of  \cref{eq:l2_loc_approx} is close to  \cref{eq:l2approx}. The estimate is quantitative and essentially independent of $m$. Indeed, the estimate only depends on the value of $C(m)$, but this is necessary (see \cref{rem:local-re}). On the other hand, if we choose the coefficients $a_i$ randomly, then $C(m) \leq C_0$ for some uniform constant $C_0$ with high probability and the estimate becomes  independent of the number of nodes $m$.

Moreover, the constants depends only on the BV-norm of $y$ which is a rather mild regularity assumption (see \cref{rem:local1} below). For this reason  \cref{thm2} can be applied also to the case when the function $y$ takes only discrete values $y(x) \in \{ 1,2, \dots\} $ which is common in classification problems. The drawback is that the proof is rather difficult and as such we obtain a rate $R^{-1/9}$ which is not optimal. Indeed, if we would assume that $C(m)$ is bounded and that $y$ is Lipschitz continuous, the proof would be simpler and it would give a better rate $R^{-\alpha}$ for $\alpha$ closer to one. As in \cref{thm1}, we give explicit values for the constants, but we do not optimize their values in order to avoid heavy computations. 

We find \cref{thm2} the deepest result of our paper and since its proof is rather technical, we give a short outline of its proof  in the next section.

\begin{remark} \label{rem:local1}
    The BV norm in \cref{thm2} is natural, since the closure of $\Hma$ in $L^2$ belongs to the  space of BV functions (see \cref{thm4}). On the other hand, we cannot replace the BV-norm  e.g. by $L^2$- or even $L^\infty$-norm of $y$. The reason for this is that every function $f_W \in \Hma$ with $|W|\leq R$ is $R$-Lipschitz continuous and it is not possible to approximate $L^2$-functions with $R$-Lipschitz functions such that the error depends only on the $L^2$-norm of $y$ (think of $\sin(kx)$ for large $k$ on $[0,1]$). From a learning perspective (classification), the BV-norm makes sense as a complexity measure, as it will "count" the number "switches" between the class $1$ and $0$ on the unit circle.
\end{remark}

\subsubsection*{Uniform approximation theorem }
\phantom{.}

Let us now finally focus on Question \cref{q:1}. As mentioned before, Question \cref{q:1} is related to the universal approximation theorem, which is a well studied problem
\cite{Cybenko,Hanin,hecht1987kolmogorov,heinecke2020refinement,HH,Hornik,KidgerLyons,kratsios2020non,kratsios2021universal,LLPS,lu2017expressive,park1991universal,park2020minimum,Pinkus,shaham2018provable,tabuada2020universal,yarotsky2018universal}.
We need to quantify the universal approximation theorem (see \cref{sec:thm3}) in order to give bounds on the size of \cref{eq:l2approx}. In the context of Barron spaces, see \cite{Barron}.
Let us begin with a simple observation that we may write the ReLU function as
\[
    \sigma(t) = \max\{t,0 \} = \frac{|t|}{2} + \frac{t}{2} = \text{symmetric} + \text{linear}.
\]
Therefore since $x \mapsto w_i\cdot x$ is linear we deduce that every function $f_W \in \Hma$ is of the form 
\[
    f_W = \text{symmetric} + \text{linear}.
\]
Therefore if we have $y \in L^2(\S^1)$, we decompose it as $y = y^s + y^a$, where $y^s(x) = y^s(-x)$ and $y^a(x)=-y^a(x)$ (antipodally symmetric and antipodally anti-symmetric). Further, we decompose $y^a = l + g$, where $l$ is linear and $g$ is orthogonal to linear functions. Then it is clear from the previous discussion, and we show this in \cref{anti-symmetric}, that $\Hma \perp g$ and therefore 
\begin{align*}
    \|f_w - g\|_{2}^2 = \|g\|_{2}^2
\end{align*}
for all $f_W \in \Hma$. We thus conclude that in order to study \cref{eq:l2approx} it is natural to assume that the function $y$ satisfies 
\begin{equation}
	\label{more than symmetric}
	y = y^{s}  +  l \qquad \text{for  symmetric } \,y^{s} \quad  \text{and linear }\, l.
\end{equation}
Note that the function $y(x)= \I\{x_2 \geq 0\} x_1$ in \cref{badexamp} can be written as
\begin{align*}
	y(x) = \frac12 \left(\I\{x_2 \geq 0\} x_1- \I\{x_2 \leq 0\} x_1\right) + \frac12 x_1
\end{align*}
and thus it satisfies \cref{more than symmetric} as $x \mapsto  \I\{x_2 \geq 0\} x_1-\I\{x_2 \leq 0\} x_1$ is (antipodally) symmetric.

Our final result completes our study of the problem 
\cref{min prob}. 
\begin{restatable}[Uniform approximation theorem]{theorem}{bvtheorem}
	\label{thm3}
	Assume that $y \in  BV(\S^1)$ satisfies \cref{more than symmetric}. Then
	\begin{align*}
		\inf_{f_W \in \Hma}\| f_W - y \|_{L^2(\S^1)}^2 \leq \frac{62  \|y\|_{BV}^2}{\underbar m},
	\end{align*}
	where $\underbar m$ is defined in \cref{def:underbar}.
\end{restatable}
\cref{thm3} quantifies the universal approximation theorem, since it states that for a given $y \in BV(\S^1)$ which satisfies \cref{more than symmetric}, we may find a function $f_W \in \Hma$ which is close to $y$ and that this error tends to zero as $\underbar m \to \infty$. It is rather clear that the convergence depends on $\underbar m$ and not on $m$. Indeed, if $y$ is a negative function and $a_i = 1$ for all $i =1, \dots, m$ then the best approximation in $\Hma$ is given by the zero function for all $m$.  

\cref{thm3} is optimal in the sense that we cannot replace the BV-norm by the $L^2$-norm of $y$. Indeed, one may see this by considering the symmetric function (in polar coordinates) $y(\theta) = \sin(2k \theta)$. It is rather clear that we cannot approximate $y$ by functions in  $\Hma$ uniformly when  $k$ tends to infinity.

\subsection{Outline of the proof of \cref{thm2}}
   
    As we mentioned above, the proof of \cref{thm2} is rather technical and therefore we  prefer to outline it here in order to highlight the key ideas. First, since $y$ is a BV-function, we may replace it by a $C^1$-function $y_R$ with an approximation argument from  \cref{lem:approximation}. We only have to be careful in controlling the error that we make by doing this. 
    
    Next we use \cref{thm4}, which we prove in \cref{sec:closure}, to find the  best approximation function $g \in \cHma$, that is
    \[
        \|g-y_R\|_2^2  = \inf_{f_W \in \Hma} \|f_W -y_R\|_2^2.
    \]
    The idea of the proof is to find a function $f_{W_R} \in \Hma$, with $|W_R| \leq C(m)R$,  which approximates $g$ well. In order to do this we need to study the regularity, or the complexity, of the minimizer $g$.
    
    In order to do this we again use \cref{thm4} which states that $g$ is a sum of simple BV-functions and, in particular, it is piecewise linear. We then use the minimality of $g$ and the characterization of its symmetric part, given by \cref{lem:symasym}, to find the Euler-Lagrange equation for the symmetric part $g^s$, which reads as   
     \begin{align} \label{eq:symmetric_euler}
        \int_S g^s x dx = \int_S y^s x dx,
    \end{align}
    where $y^s$ is the symmetric part of $y$ and $S$ is any arc where $g$ is linear. We note that \cref{eq:symmetric_euler} is in fact  a system of two equations (recall that $x \in \S^1 \subset \R^2$) which contain information on the zeroth and on the first order behavior of $g^s$ respectively. In \cref{lem:eulersecbound} we use an argument inspired by regularity theory of elliptic PDEs to deduce Lipschitz estimates for $g^s$ in terms of the $C^1$-norm of $y$. In \cref{lem:ubound} we turn this estimate into  information on the weights which define $g$ and thus we are able to bound the complexity of $g$. This is the core of the proof, and after that we may approximate $g$ in a rather straightforward way by  $f_{W_R} \in \Hma$, with $|W_R| \leq C(m)R$.

    \begin{remark}
        \label{restr:lebesgue}
        The proof of \cref{thm2} relies on the fact that the Euler-Lagrange equation \cref{eq:symmetric_euler} contains information for all small arcs $S$ on the unit circle. If we try to generalize \cref{thm2} to more general class of probability measures $\mu$, then $\mu$ may not have support on all small arcs and our proof breaks down. What we need for the proof is that the measure is mutually absolutely continuous w.r.t. the Lebesgue measure in a quantifiable sense. We believe for instance that the class of $A^\infty$ measures would suffice, see \cite{CF,GR}. We also need  the measure to be  antipodally symmetric in order to  obtain \cref{eq:symmetric_euler}. 
    \end{remark}

\section{Preliminaries}
In this short section we introduce our notation, recall basic results on log-Sobolev inequalities and recall the definition and some basic facts on BV-functions.

\subsection{Notation}

Whenever we have a point $x \in \R^n$ we denote $|x|$ as the Euclidean $2$-norm. We denote $\S^1 \subset \R^2$ as the unit circle and whenever we have a set $S \subset \S^1$ we denote with $|S|$ the normalized Lebesgue length of that set, i.e.
\[
    |S| = \fint_{S} dx.
\]
For a discrete set of indices i.e. $I \subset \N$ we denote with $|I|$ the number of elements in $I$.

We  define the $L^2$ and $L^\infty$ on norms on $\S^1$ as 
\[
    \|f\|_{2}^2 := \fint_{\S^1} f^2(x)\, dx \qquad \text{and} \qquad \|f\|_{\infty} := \text{ess sup}_{x \in \S^1} |f(x)|, 
\]
and the $C^1$-norm of  as 
\[
    \|f\|_{C^1} :=  \|f\|_{\infty} + \|f'\|_{\infty}. 
\]

We recall the definition of the ReLU activation function as
\begin{align*}
    \sigma(t) = \max\{t,0\}, \quad t \in \R.
\end{align*}
It has two important properties. First, it is homogeneous $\sigma(\lambda t) = \lambda \sigma(t)$ for $\lambda >0$ and its derivative is the Heaviside step function $\sigma'(t) = \I\{t \geq 0\}$, where $\I$ denotes the indicator function.

The function space that we consider in this paper is the following:
\begin{definition} \label{def:NN}
    Consider an integer  $m > 0$, index set $I = \{1,\dots, m \}$ and coefficients $a=(a_i)_{i =1}^m \in \{-1,1\}^m$. Then the set of all functions of the form
    \begin{equation*}
    	f_{W,a}(x) = \frac{1}{\sqrt{m}}\sum_{i \in I} a_i \sigma(w_i \cdot x), \quad w_i \in \R^2,\, x\in \S^1
    \end{equation*}
    is denoted by $\Hma$. We sometimes simplify the notation by $f_W = f_{W,a}$, since the coefficients are fixed.  We define further $I_+ = \{ i \in I : a_i = 1\}$, $I_- = \{ i \in I : a_i = -1\}$  and  recall that  
    \begin{align*}
	    \underbar m = \min \{  |I_+| ,   |I_-|  \}.
    \end{align*}
\end{definition}

\begin{remark} \label{rem:reorder}
    In the rest of the paper we will assume that $\underbar m =  |I_-| \leq  |I_+|$. We may do this without loss of generality, up to a changing sign.

    Furthermore, we will consider the following reordering of the coefficients  $a = (a_1,\ldots,a_m)$ such that the first $2 \underbar m$ of the $a_i$'s alternate in sign, and the remaining $a_i$'s are positive. This allows us to consider restrictions of the minimization problem from $\Hma$ to $\mathcal{H}_{m',a'}$ where $a' = (a_1,\ldots,a_{m'}) \in \{-1,1\}^{m'}$ and $0 < m' < m$. The notation that we will adhere to is
    \begin{align} 
         \inf_{f \in \mathcal{H}_{m',a'}} \|f - y\|_2^2 \label{def:phiprim}
    \end{align}
\end{remark}

Finally we define the following  function space 
\begin{align}\label{L2al}
	\Lal := \{f: f \in L^2, \text{ anti-symmetric part of $f$ is linear} \}.
\end{align}
This is a natural space, since by our earlier discussion before \cref{thm3} it is clear that  $\Hma \subset \Lal$. Note also that $\Lal^\perp$ consists of anti-symmetric functions.

\subsection{Log-Sobolev inequalities}
In this section  we list results related to the Log-Sobolev inequality   and Poincar\'e inequality (or Spectral gap)   that  we need in the proof of  \cref{thm1}. As an introduction to the topic we refer to \cite{Ledoux} and \cite{VBook}.

Let  $d \mu$ be a probability measure on $\R^N$. Then we say that $\mu$ satisfies the Log-Sobolev inequality (LSI) with constant $C_{LS}$ if for all Lipschitz function $u$ with $\int u^2 d\mu = 1$ the following holds
    \begin{align} \label{LSI}
    \int u^2 \log u^2 d\mu \leq 2 C_{LS} \int |\nabla u|^2 d\mu.
    \end{align}
 We say that $\mu$ satisfies  Poincar\'e inequality with constant $C_{P}$ if for all Lipschitz function $u$ with $\int u d\mu = 0$ the following holds
    \begin{align} \label{poincare}
    \int u^2  d\mu \leq  C_{P} \int |\nabla u|^2 d\mu.
    \end{align}
    It is well known that LSI implies the Poincar\'e inequality \cite{Ledoux, VBook}.
\begin{remark} \label{rem:LSI-poincare}
	If LSI holds for a constant $C_{LS}$ then the Poincar\'e inequality holds with the constant $C_P = C_{LS}$.
\end{remark}

We recall the standard result on the Gaussian measure.  
\begin{lemma}\cite{Gross,Ledoux,VBook} 
	\label{lem:logsob}
	The Gaussian measure $d\mu_\Lambda = c_{n,\Lambda} e^{-\frac{\Lambda \|x\|^2}{2}} dx$, where $c_{n,\Lambda}$ is the normalization factor,  satisfies  the Log-Sobolev inequality (LSI) \cref{LSI} with constant 
	\[
	C_{LS} = \frac{1}{\Lambda}.
	\]
\end{lemma}

By \cref{lem:logsob}  and  \cref{rem:LSI-poincare} we know that the Gaussian measure satisfies the Poinca\'e inequality \cref{poincare} with a constant that is independent of the dimension. This is crucial for us, since our goal is to have estimates which are dimension-free.  
The following well-know result enables us to prove LSI  for measures that are perturbation of the Gaussian measure.  

\begin{lemma}[Holley-Stroock perturbation lemma]\label{lem:hostro}
	If a LSI holds for the measure $d\mu$ with constant $C_{LS}$,  then the perturbed measure
	\begin{align*}
	    Z^{-1} \exp(G(x))d\mu(x)
	\end{align*}
	where $Z$ is a normalization constant, satisfies a LSI with constant $C_{LS} \exp(\osc(G))$.
\end{lemma}

The advantage of \cref{lem:hostro} for us is that it can be applied to general class of measures. The disadvantage  is that in our setting it provides rather poor bound on the Poincar\'e constant $C_P$. We will use the following perturbation result by Aida \cite{Aida}, which fits well in our setting.   

\begin{lemma}[Aida perturbation lemma]\label{lem:aida}
	Assume LSI holds for the measure $d\mu$ with constant $C_{LS}$  and assume that for a function $G$ and constant $\beta \in (0,1/4]$ the following holds
	\begin{align}
		\label{eq:lemaida1}\int \exp(2C_{LS}(1+\beta) |\grad G|^2) d\mu &\leq 1 + \frac{\beta}{2} \\
		\label{eq:lemaida2}\int \exp(G) d\mu &\geq 1-\frac{\beta}{2} \\
		\label{eq:lemaida3}\int \exp(2G) d\mu &= 1.
	\end{align}
	Then the perturbed measure $\exp(2G(x))d\mu(x)$, satisfies  Poincaré inequality with constant
	\begin{align*}
		\frac{64 C_{LS} (1+\beta)(1-(\beta/2))}{\beta^2}.
	\end{align*}
\end{lemma}

\subsection{BV Functions} \label{subsec:bv}

We begin by recalling the definition and basic facts on functions of bounded variation, or simply  BV-functions. As an introduction to the topic we refer to \cite{AFP}.

Given a function $y \in \S^1 \to \R$, there is a natural identification of $y$ as a $2\pi$-periodic function $\tilde y: \R \to \R$  defined as $\tilde y(\theta) = y(\cos \theta, \sin \theta)$. We will assume this identification from now on.
We recall that a $2\pi$-periodic function $y : \R \to \R$ is a function of bounded variation if
\begin{align*}
    \sup \left(  \sum_{i=1}^{n-1} |y(\theta_{i+1}) - y(\theta_i)| :\,  n \geq 2, \, 0 = \theta_1 < \theta_2 <   \dots < \theta_n = 2 \pi   \right) < \infty.
\end{align*}
We denote the space of BV-functions on $\S^1$ by $BV(\S^1)$.  An equivalent definition is to require that the total variation of $y$ on $(0, 2 \pi)$ 
\begin{equation} \label{variation1}
    V_y :=  \sup \left(  \fint_{(0,2 \pi)} y(\theta) \varphi'(\theta) \, d \theta :  \, \varphi \in C^1_{per}(\R) , \, \|\varphi\|_{L^\infty} \leq 1 \right)
\end{equation}
is bounded.
The derivative of the $BV$-function $y$ is a finite Radon measure which we denote by $\mu_y$ and we have
\begin{align*}
    \int_{(0,2 \pi)} y(\theta) \varphi'(\theta) \, d \theta =-  \int_{(0,2 \pi)} \varphi(\theta) \, d \mu_y
\end{align*}
for all $\varphi \in C^1_{per}(\R) $. It is trivial to see that any BV-function is bounded and we define  its BV-norm as
\begin{align*}
    \|y\|_{BV} := \|y\|_{\infty} + V_y.
\end{align*}

If $y$ is a BV-function then it is, in particular, $L^2$-function and we may write  it in terms of  Fourier series 
\begin{equation*}
	y(\theta) = \sum_{k= 0}^\infty a_k \sin(k \theta) + b_k \cos(k \theta).
\end{equation*}
By choosing $\varphi(\theta) = -\cos(k \theta)$  we have by $\varphi'(\theta) = k \sin(k \theta)$ and \cref{variation1} that
\begin{align*}
	\|y\|_{BV} \geq  k \fint_{(0,2 \pi)} y(\theta)  \sin(k \theta)\, d \theta  = k \fint_{(0,2 \pi)} a_k \sin^2(k \theta)\, d \theta = \frac{k a_k}{2}.
\end{align*}
By choosing $\varphi(\theta) = \cos(k \theta) $ we obtain by the same argument $- \frac{k a_k}{2} \leq \|y\|_{BV}$, and by choosing $\varphi(\theta) = \pm \sin(k \theta) $ we obtain the same bound for $b_k$. Hence, we conclude
\begin{equation} \label{variation2}
	|a_k|,|b_k| \leq 2\frac{\|y\|_{BV}}{k}.
\end{equation}

We may use  \cref{variation2} to approximate BV-functions with smooth functions in a quantitative way.

\begin{lemma} \label{lem:approximation}
	Assume $y \in BV(\S^1)$ and fix $r \in \N_+$. Then there is $ y_r \in C^\infty(\S^1)$ such that  $\|y_r\|_{\infty} \leq   \|y\|_{\infty}$,  $\|y_r\|_{C^1} \leq 5 r \, \|y\|_{BV} $  and
	\[
		\|y - y_r\|^2_2 \leq 16 \|y\|_{BV}^2\frac{1}{r}.
	\]
\end{lemma}

\begin{proof}
	Let $u$ be the solution of the heat equation $\partial_t u = u''$ with initial datum $u(\theta,0)= y(\theta)$. We  define $y_r(\theta) := u(\theta,r^{-2})$. Let us check that this satisfies all the required conditions. First, by the maximum principle, $\|y_r\|_{\infty} \leq  \|y\|_{\infty}$.

	We proceed by  writing $y$ in Fourier series as
	\begin{align*}
		y(\theta) = \sum_{k=0}^\infty a_k \sin(k \theta) + b_k \cos(k \theta)
	\end{align*}
	and the solution to the heat equation as 
	\begin{align} \label{eq:heat_smooth}
		u(\theta,t) = \sum_{k=0}^\infty a_k \sin(k \theta) e^{-k^2 t} + b_k \cos(k \theta) e^{-k^2 t}.
	\end{align}
	By differentiating \cref{eq:heat_smooth} and using \cref{variation2} we obtain 
	\begin{align*}
		|\frac{\partial}{\partial \theta} u(\theta,t)| &\leq \sum_{k=1}^\infty  (|a_k |  k  + |b_k |  k) e^{-k^2 t}  \\
		&\leq 4 \|y\|_{BV} \sum_{k=1}^\infty e^{-k^2 t} \\
		&\leq 4 \|y\|_{BV} \int_{0}^\infty e^{-x^2 t} \, dx \leq \frac{2 \sqrt{\pi}}{\sqrt{t}} \|y\|_{BV}.
	\end{align*}
	Hence $|y_r'(\theta)| =  |\frac{\partial}{\partial \theta} u(\theta,r^{-2})| \leq 4 r \, \|y\|_{BV}$ and thus $\|y_r\|_{C^1} = \|y_r\|_{\infty} + \|y_r'\|_{\infty} \leq 5r \|y\|_{BV}$ for $r \geq 1$.

	We continue by estimating $	\|y - u(\cdot, t)\|^2_2$ for $t = r^{-2}$ by using the Fourier series 
	\begin{align*}
		\fint_{0}^{2\pi} &|y(\theta) - u(\theta, t)|^2 \, d\theta  \leq \sum_{k=0}^\infty | a_k \left (1-e^{-k^2 t} \right )|^2 + | b_k \left (1-e^{-k^2 t} \right )|^2\\
		&\leq \sum_{k=1}^{r} |a_k \left (1-e^{-k^2 t} \right )|^2 + | b_k \left (1-e^{-k^2 t} \right )|^2 +\sum_{k=r+1}^\infty (a_k^2+ b_k^2) .
	\end{align*}
	We first use \eqref{variation2} the estimate the second term as 
	\begin{align*}
		\sum_{k=r+1}^\infty (a_k^2+ b_k^2)\leq 8\|y\|_{BV}^2 \sum_{k=r+1}^\infty \frac{1}{k^2} \leq
		\frac{8\|y\|_{BV}^2}{r}.
	\end{align*}
	We use \eqref{variation2} to estimate the first term as 
	\begin{align*}
		 &\sum_{k=1}^{r} |a_k \overbrace{\left (1-e^{-k^2 t} \right )}^{\leq k^2 t}|^2 + | b_k \left (1-e^{-k^2 t} \right )|^2\\
		 &\leq \sum_{k=1}^{r} |t \,   k^2  a_k |^2 + |t \,   k^2  b_k |^2 \\
		 &\leq 8\|y\|_{BV}^2 t^2 \sum_{k=1}^{r} k^2 =  8  \|y\|_{BV}^2 t^2 \frac{r(r+1)(2r +1)}{6} \\
		 &\leq 8  \|y\|_{BV}^2  t^2 r^3 \leq 8 \|y\|_{BV}^2\frac{1}{r} 
	\end{align*}
	for $t = r^{-2}$.
	Hence we deduce $\|y - u(\cdot, r^{-2})\|_2^2 \leq \frac{16 \|y\|_{BV}^2}{r}$.
\end{proof}

 Recall that every $f_W \in \Hma$ belongs to the function space $\Lal$ defined in \cref{L2al}. It is therefore natural to decompose also $y = y_1 + y_2$ where $y_1 \in \Lal$ and $y_2 \in \Lal^\perp$. In the following lemma  we make the rather obvious remark that $y$ bounds  the norm of $y_1$.

\begin{lemma} \label{lem:lal}
	Assume $y \in BV(\S^1)$ and let us write $y = y_1 + y_2$ with $y_1 \in \Lal$ and $y_2 \in \Lal^{\perp}$, where $\Lal$ is defined in \cref{L2al}. Then
    \[
        \|y_1\|_{BV} \leq 4 \|y\|_{BV}.
    \]
\end{lemma}

\begin{proof}
	Since $y_1 \in \Lal$ it can be written as $y_1(x) = y^s(s) + l(x)$,  where $y^s$ is the symmetric part of $y$ and $l$ is a linear function. First, we may write $y^s(x) = \frac{y(x) + y(-x)}{2}$ and therefore
	\[
	    \|y^s\|_{BV} \leq \|y\|_{BV}.
	\]
	Since $y^s \perp l$ and $y_2 \perp l$ we have
	\begin{align*}
		\fint_{\S^1} y(x) \frac{l(x)}{\|l\|_{\infty}} dx = \|l\|_{\infty} \fint_{\S^1} \left (\frac{l}{\|l\|_{\infty}} \right )^2 dx = \frac{\|l\|_{\infty}}{2},
	\end{align*}
	where the last equality follows from the fact that $\fint_{(0,2\pi)} \sin^2(\theta) d\theta = 1/2$.
	From the above equality and from Cauchy-Schwarz we obtain
	\begin{align*}
		\|l\|_{\infty} = 2\fint_{\S^1} y(x) \frac{l(x)}{\|l\|_{\infty}} dx \leq 2\frac{\|y\|_{2}}{\sqrt{2}}.
	\end{align*}
\end{proof}

\section{Convergence to equilibrium and properties of the cost function}
\label{sec:thm1}

In this section we assume $\mu$ to be  merely a probability measure on $\S^1$ and denote 
\[
    \E_{\mu}[f] := \int_{\S^1} f(x) \, d\mu(x) \qquad \text{and} \qquad \|f\|_{L^2(\mu)}^2 := \E_{\mu}[f^2].
\]
Our goal  is to   prove \cref{thm1}.  To this aim we need first to study the structure of the  cost function $\Phi$ defined as
\[
    \Phi(W) := \E_{\mu}[(f_W - y)^2] = \|f_W - y\|_{L^2(\mu)}^2,
\]
where we write $f_W  = f_{W,a} \in \Hma$ for short.  We   first study the growth bounds from below and above of the cost function.
\begin{lemma}
\label{growth&coer}
	Assume that $ \|y\|_{L^2(\mu)} \leq 1$. Then,
	\begin{align*}
		\Phi(W) \leq (|W| + 1)^2
	\end{align*}
and
	\begin{align*}
		|\grad_W \Phi(W)|^2 \leq 4 \Phi(W)  \quad \text{and} \quad  \langle \grad_W \Phi(W), W \rangle  \geq \Phi(W) - 1.
	\end{align*}
\end{lemma}
\begin{proof}
	By Cauchy-Schwarz we have for every $f_W \in \Hma$ and every $x \in \S^1$
	\begin{align}
		\notag |f_{W}(x)| &= \frac{1}{\sqrt{m}}| \sum_{i=1}^m a_i \sigma(w_i \cdot x) |\\
		&\leq \frac{1}{\sqrt{m}} \left( \sum_{i=1}^m \underbrace{a_i^2}_{=1}  \right)^{1/2} \left( \notag \sum_{i=1}^m \underbrace{\sigma(w_i \cdot x)^2}_{\leq |w_i|^2} \right)^{1/2}\\
		\label{eq:fbound} &\leq  \frac{1}{\sqrt{m}} \left( m \right)^{1/2}  \left( \sum_{i=1}^m |w_i|^2 \right)^{1/2} = |W|.
	\end{align}
	Using \cref{eq:fbound} and Cauchy-Schwarz again yields
	\begin{equation}
		\label{eq:covbound}
		\big{|}\E_\mu[f_W y]\big{|} \leq \|f_W\|_{L^2(\mu)} \|y\|_{L^2(\mu)}
		\leq \|f_W\|_{L^2(\mu)}	\leq |W|.
	\end{equation}
	By  \cref{eq:fbound,eq:covbound} we may now bound the value of the cost function 
	\begin{multline}
		\label{eq:lossbound}
		\E_\mu[(f_W-y)^2] = \E_\mu[f_W^2]-2 \E_\mu[f_W y] + \E_\mu[y^2] \\
		\leq |W|^2 + 2 |W| + 1
	\end{multline}
	which yields the first claim.

	By differentiating $f_{W,a}(x)$ with respect to the weight $w_i$ yields 
	$$
	    \frac{\pa}{\pa w_i} f_{W,a}(x) = \frac{1}{\sqrt{m}}a_i \sigma'(w_i \cdot x) x.
	$$ 
	Therefore
	\begin{align}
		\label{eq:gradfbound}
		\|\grad_W f_W\|_{L^2(\mu)}^2 = \E_\mu \left [\frac{1}{m} \sum_{i=1}^m a_i^2 |x|^2 \sigma'(w_i \cdot x)^2 \right ] \leq 1.
	\end{align}
	Hence, we deduce from \cref{eq:lossbound,eq:gradfbound} that
	\begin{align*}
		|\grad_W \Phi(W)|  &= |\grad_W \E_\mu[(f_W-y)^2]| \leq 2 \E_\mu [|\grad_W f_W (f_W-y)|] \\
		&\leq 2 \|\grad_W f_W\|_{L^2(\mu)}\|f_W-y\|_{L^2(\mu)}\\
		&\leq 2 \sqrt{\Phi(W)},
	\end{align*}
	which yields the second claim.

	In order to prove the third claim we observe that by homogeneity, $f_{\lambda W}(x) = \lambda f_W(x)$ for all $\lambda >0$. Therefore we have
	\begin{align*}
		\langle \nabla f_W(x), W \rangle = \frac{\pa}{\pa \lambda} \Big{|}_{\lambda = 1} f_{\lambda W}(x) = f_W(x).
	\end{align*}
	Thus we deduce
	\begin{align*}
	\langle \grad_W \Phi(W), W \rangle  &= \langle \nabla \E_\mu [(f_W-y)^2], W \rangle \\
	&= 2\E_\mu [\langle \nabla f_W , W \rangle (f_W-y)] \\
	&= 2\E_\mu [(f_W-y)^2] + 2\E_\mu [y(x)(f_W-y)] \\
	&\geq \E_\mu [(f_W-y)^2] -\E_\mu [y^2] \\
	&\geq \Phi(W) - 1 .
	\end{align*} 
\end{proof}

One might be tempted to think that the third inequality in  \cref{growth&coer}  implies   the cost function to be  coercive.  This is however not true, as the next lemma shows (see also \cref{badexamp}). In particular,  we show that the gradient flow diverges for certain starting points. 

\begin{lemma}
\label{lem:badexamp2}
	Assume that $y(x) =  \I\{x_2 \geq 0\} x_1$, $a_1 = -a_2 =1$ and  $\mathcal{H}_{2,a}$ and $\Phi$ are as in \cref{badexamp}. Choose  $W_0 = (w_1, w_2)$ with $w_1 =(1/2,b)$ and $w_1 =(-1/2,b)$ for $b \geq 1$. Then the gradient flow $(W_t)_{t \geq 0}$ of $\Phi$ starting from $W_0$ diverges, i.e. 
	\[
	    \lim_{t \to \infty} |W_t| = \infty.
	\]
\end{lemma}
The claim follows from a rather simple calculation which we provide for the convenience of the reader. 
\begin{proof}
    Let us show that the gradient flow $(W_t)_{t \geq 0}$,
    \begin{equation*} 
        \frac{d}{dt} W_t = - \nabla_W \Phi(W_t),
    \end{equation*}
    is always of the form $W_t = (w_1(t), w_2(t)) \in \R^4$ with 
    \begin{equation} \label{badexmp2_2}
        w_1(t) = (1/2, b_t) \qquad \text{and} \qquad w_2(t) = (-1/2, b_t),
    \end{equation}
    with $b_t \geq 1$. Recall that we assume  $\Phi(W) = \|f_w - y\|_2^2$.

    Let us assume that $W_t = (w_1, w_2)$ is  of the form \cref{badexmp2_2} and let us calculate $\nabla_W \Phi(W_t)$. We denote $S_{12} =  \{x \in \S^1: w_1 \cdot x \geq 0 \} \setminus  \{x \in \S^1: w_2 \cdot x \geq 0 \}$ and $S_{21} =  \{x \in \S^1: w_2 \cdot x \geq 0 \} \setminus  \{x \in \S^1: w_1 \cdot x \geq 0 \}$  and have by a straightforward calculation
    \[
        \partial_{w_1} \Phi(W) \cdot e_1= \frac{2}{2 \pi } \int_{S_{12}} \left(  \frac12 x_1 + b_t x_2 - \I_{\{x_2 \geq 0\}} x_1\right)x_1 \, dx = 0
    \]
    and similarly $\partial_{w_2} \Phi(W) \cdot e_1= 0$. Moreover, 
    \[
        \begin{split}
            \partial_{w_1} \Phi(W) \cdot e_2&= \frac{2}{2 \pi } \int_{S_{12}} \left(  \frac12 x_1 + b_t x_2 - \I_{\{x_2 \geq 0\}} x_1\right)x_2 \, dx \\
            &= \frac{2}{2 \pi } \int_{S_{21}} \left(  \frac{1}{2} x_1 + b_t x_2 - \I_{\{x_2 \geq 0\}} x_1\right)x_2 \, dx \\
            &= \partial_{w_2} \Phi(W) \cdot e_2.
        \end{split}
  \]
  This means that if $W_t$ is of the form $W_t = (1/2, b_t, -1/2, b_t) \in \R^4$ then the gradient is of the form $\nabla_W \Phi(W_t) = (0,c_t,0,c_t)$. This proves \cref{badexmp2_2}. The argument also yields that $|\nabla_W \Phi(W_t)|^2 = 2c_t^2 >0$ when  $|b_t| < \infty$. Therefore we deduce that necessarily $|b_t| \to \infty$ and the claim follows. 
\end{proof}

The fact that the gradient flow may diverge to infinity as shown by  \cref{badexamp} and \cref{lem:badexamp2}  forces us to consider a penalizing term to confine the gradient descent. We recall that we define the penalization of $\Phi$ as
\begin{equation}
    \label{pena again}
    \Phi_R(W) = \max\{ \Phi(W), 4(|W|^2 -R^2) \}.
\end{equation}
The reason we choose this penalization is that it is Gaussian outside a large ball and we wish to utilize the contracting effect of high dimension, where the local non-convexity of $\Phi$ matters less and less. As such, when the dimension is large w.r.t. $R$ we may expect to have Poincar\'e inequality with a constant  independent of the dimension.

\cref{thm1} concerns the density of the stochastic gradient flow \eqref{SDE2}. It is well known  that the density $\rho$ solves the following Fokker-Planck equation 
\begin{equation*}
	\frac{\partial \rho}{\partial t} = \nabla \cdot \left ( \rho \nabla \Phi_R + \eps^2 \nabla \rho \right ).
\end{equation*}
By setting $u(t,x) = \rho(t,x)e^{\Phi_R(x)/\eps^2}$  we obtain a formulation
\begin{equation}
    \label{FP3}
    \frac{\partial u}{\partial t} = \eps^2 \Delta u - \nabla \Phi_R \cdot \nabla u
    = \nabla\cdot \left (\eps^2 \nabla u e^{-\Phi_R/\eps^2} \right )
\end{equation}
In order to quantify the convergence of $u$, and thus in turn of $\rho$, we need to prove the Poincar\'e inequality (PI) \cref{poincare} for the measure  $e^{-\Phi_R/\eps^2} dW$. This is the most important result of the section, and deserves some comments. 

The history of Poincar\'e inequalities for exponential measures is  rich and deep. In the Gaussian case, i.e. $e^{-|x|^2} dx$, the PI is by \cref{rem:LSI-poincare}  a consequence of the log-Sobolev inequality (LSI), which is also  called the Gross logarithmic Sobolev inequality, see \cite{Gross}. As mentioned before, the crucial  fact is that  the Gaussian Poincar\'e constant is independent of the dimension.

Since the cost function $\Phi$ defined  \cref{def:costfunc} and $\Phi_R$ defined in \cref{pena again} are non-convex, we  cannot use the well developed theory of  PI for log-convex measures \cite{VBook}. In order to obtain the PI for measures that are not log-convex, the most well known methods are either  by  perturbation argument or via Lyapunov functions. In the first case, the idea is to perturb the PI related to a log-concave measure, for which the PI or even LSI holds.  In this case  the Holley-Stroock perturbation result, \cref{lem:hostro}, is essentially sharp for local perturbations, if  we know only the oscillation. For global perturbations, the works  \cite{Aida,Aida2} are for  our knowledge the best result which gives dimensional-free and explicit constants. The other method, using Lyapunov functions, does not seem to be stable w.r.t. the dimension even for log-concave measures.
It is worth mentioning that if we know more about the geometry of our local perturbation, we could use the tunneling estimates in \cite{Sjostrand} to obtain sharper constants.

As for  now we do not know enough about the geometry of local minimas of the cost function to apply other methods than perturbation. As such, our proof consists of a combination of the Holley-Stroock perturbation method together with the one by Aida. Finally, it is essential to remark that we  trace the constants explicitly.

\begin{proposition}
	\label{poincare thm1}
	Assume that $0 < \eps \leq 1$ and $R \geq 10$. Then,
	\begin{align*}
		\int_{\R^{2m}} u^2 \, e^{-\Phi_R(W)/\eps^2}\, d W \leq C_P\, \eps^2 \int_{\R^{2m}} | \nabla u|^2 \, e^{-\Phi_R(W) /\eps^2}\, d W
	\end{align*}
	for every Lipschitz function with $\int_{\R^{2m}} u \, e^{-\Phi_R(W) }\, d W = 0$.  When $m \geq 24 R^2\eps^{-2} $ the constant is bounded by
	\begin{align*}
		C_P \leq 140.
	\end{align*}
	Otherwise we have the bound $C_P \leq \frac{1}{8} e^{(4R^2  +2)\eps^{-2}}$.
\end{proposition}

\begin{proof}
	We will use the perturbation results  \cref{lem:hostro} and  \cref{lem:aida} to prove the Poincar\'e inequality for the measure $e^{-\Phi_R(W)/\eps^2} dW$, where $\Phi_R(W) $ is defined in \cref{pena again}. We may write $\Phi_R(W) = V(W) - 2F(W)$, where  $V(W) = 4(|W|^2 -R^2)$ and
	\begin{align*}
		F(W) = \frac12(V(W) - \Phi(W))\I_{\{\Phi(W) > V(W)\} }.
	\end{align*}
	Since 
	\begin{align*}
	    e^{-V(W)/\eps^2} = e^{4R^2\eps^{-2}} e^{-4 \eps^{-2}|W|^2},    
	\end{align*}
	the reference measure is the scaled Gaussian measure $d \gamma = c_{2m}e^{-4 \eps^{-2}|W|^2}  dW$, where the normalization factor is
	\begin{align*}
		c_{2m} = \left(\frac{4}{\eps^2 \pi}\right)^{m},
	\end{align*}
	and the perturbed measure is $d \gamma_F = e^{2\eps^{-2}F} d\gamma$. By \cref{lem:logsob} $\gamma$ satisfies the log-Sobolev inequality   with constant $C_{LS} = \frac{\eps^2}{8}$.

	Let us first show that
	\begin{equation}
		\label{poincare st1}
		F(W) = 0 \qquad \text{for all} \quad  |W|^2 \geq 2 R^2.
	\end{equation}
	Indeed, by  \cref{growth&coer} we have $\Phi(W) \leq (|W| + 1)^2$. When $R \geq 10$ and $|W|^2 \geq 2 R^2$ then
	\begin{align*}
		V(W) = 4(|W|^2 -R^2) \geq (|W| + 1)^2.
	\end{align*}
	Therefore $V(W) \geq \Phi(W)$ for all  $|W|^2 \geq 2 R^2$ and thus  we have \cref{poincare st1}.

	By the definition of $F$ we have $F \leq 0$.  From \cref{growth&coer}  we deduce $4|W|^2 - \Phi(W) \geq 2|W|^2 - 2 \geq -2$. Therefore we have
	\begin{equation*}
		2|F(W)| =  (V(W) - \Phi(W))_- = (4|W|^2 - \Phi(W) - 4R^2)_-  \leq 4R^2 + 2.
	\end{equation*}
	By \cref{lem:hostro} we deduce that the Poincar\'e inequality holds for $d\gamma_F$ with a constant $C_{P,\eps} $, which we may estimate as
	\begin{equation}
		\label{poincare st2}
		C_{P,\eps} \leq \frac{\eps^2}{8}e^{(4R^2 + 2)\eps^{-2}},
	\end{equation}
	setting $C_P = \eps^{-2} C_{P,\eps}$ yields the first result.

	Let us improve the bound \cref{poincare st2} in the case  $m \geq 24 R^2\eps^{-2}$, by using Aidas result, \cref{lem:aida}. We choose $\beta = \frac{1}{4}$ in \cref{lem:aida}. Let us verify the conditions \cref{eq:lemaida1,eq:lemaida2,eq:lemaida3}. We first claim that
	\begin{equation}
		\label{poincare st3}
			\int_{\R^{2m}} e^{2C_{LS}(1 +\frac{1}{4}) | \nabla (\eps^{-2}F) |^2 }  \,d \gamma \leq 1 + \frac{1}{8},
	\end{equation}
	which gives \cref{eq:lemaida1}.
	To this aim we observe that by \cref{poincare st1}, $F= 0$ outside the ball $B_{\sqrt{2}R}$. Moreover by differentiating we get $2\nabla F(W) = \nabla V(W) - \nabla \Phi(W) = 8 W - \nabla \Phi(W) $, when $F(W) \neq 0$. Therefore we estimate by \cref{growth&coer}
	\begin{align*}
		|\nabla F(W)|^2 &= \frac{1}{4} \left(64 |W|^2 - 16  \langle \nabla \Phi(W), W \rangle + | \nabla \Phi(W)|^2\right) \\
		&\leq  \frac{1}{4} \left(64 |W|^2 - 16 \Phi(W) + 16 + 4 \Phi(W) \right) \\
		&\leq   16|W|^2 + 4.
	\end{align*}
	 Therefore we have
	\begin{align}
		\notag \int_{\R^{2m}} e^{2C_{LS}(1 +\frac14) | \nabla (\eps^{-2}F) |^2 }  \,d \gamma &= \int_{\R^{2m}} e^{\frac{\eps^2}{4}(1 +\frac14) |\eps^{-2}\nabla F|^2 }  \,d \gamma =   \int_{\R^{2m}} e^{\frac{5}{16} \eps^{-2} |\nabla F|^2 }  \,d \gamma  \\
		\notag &\leq  1 + \left(\frac{4}{\eps^2 \pi}\right)^{m}  \int_{B_{\sqrt{2} R}} e^{5 \eps^{-2}|W|^2 + \frac54 \eps^{-2} } e^{-4 \eps^{-2} |W|^2} \, d W\\
		\notag &\leq 1 + e^{\frac54 \eps^{-2}} \left(\frac{4}{\eps^2\pi}\right)^m  \int_{B_{\sqrt{2} R}} e^{\eps^{-2}|W|^2}  \, d W\\
		\label{eq:aida1} &\leq 1 +  e^{\frac54 \eps^{-2}}  \left(\frac{4}{\eps^2 \pi}\right)^m e^{2R^2 \eps^{-2}} |B_{\sqrt{2} R}|.
	\end{align}
	We continue  by recalling the volume of the $2m$-dimensional ball of radius $\sqrt{2} R$ and using Stirlings formula
	\begin{equation}
		\label{poincare ball}
		|B_{\sqrt{2} R}| = \frac{\pi^m}{m!} (\sqrt{2}R)^{2m} \leq \frac{1}{\sqrt{2\pi m}} \left( \frac{2\pi e R^2}{m}  \right)^m.
	\end{equation}
	Therefore, using \cref{eq:aida1}, \cref{poincare ball} and the assumptions $m \geq 24 R^2 \eps^{-2} $, $R \geq 10$, we deduce
	\begin{align*}
			\int_{\R^{2m}} e^{2C_{LS}(1 +\frac14) | \nabla (\eps^{-2}F) |^2 }  \,d \gamma &\leq 1 +  \frac{ e^{\frac54  \eps^{-2}}  }{\sqrt{2\pi m}} e^{2R^2 \eps^{-2}}  \left(  \frac{8 e R^2}{\eps^2 m}  \right)^m \\
		&\leq 1 +  \frac{1}{100}\left(  \frac{e^{\frac{1}{11}} e}{3}  \right)^m\\
		&\leq 1 + \frac{1}{8},
	\end{align*}
	since $\frac{e^{\frac{1}{11}} e}{3} < 1$. Hence, we have \cref{poincare st3}.

	Let us define $G_\eps := \eps^{-2}F + \alpha$, where $\alpha \in \R$ is chosen such that $\int_{\R^{2m}} e^{2G_\eps}\, d \gamma = 1$, in order to satisfy \cref{eq:lemaida3}. Note that   $e^{2G_\eps}  \,d \gamma$ is merely $c_n e^{2\eps^{-2} F }  \,d \gamma$, where $c_n$ is the normalization factor.   We claim that 
	\begin{equation}
		\label{poincare st4}
		\int_{\R^{2m}} e^{G_\eps}\, d \gamma \geq 1 - \frac18,
	\end{equation}
	which establishes \cref{eq:lemaida2}. 
	From $F \leq 0$ and from the choice of $\alpha$ it follows
	\begin{align*}
		\int_{\R^{2m}} e^{G_\eps}\, d \gamma  = e^\alpha \int_{\R^{2m}} e^{\eps^{-2}F}\, d \gamma \geq e^{-\alpha} \int_{\R^{2m}} e^{2G_\eps}\, d \gamma =  e^{-\alpha}.
	\end{align*}
	Now since
	\begin{align*}
		e^{-\alpha} = \left( \int_{\R^{2m}} e^{2\eps^{-2}F}\, d \gamma\right)^{1/2}.
	\end{align*}
	the claim \cref{poincare st4} follows once we prove
	\begin{equation}
		\label{poincare st5}
		\int_{\R^{2m}} e^{2\eps^{-2}F}\, d \gamma \geq 1 - \frac{1}{10},
	\end{equation}
	because $\sqrt{1 - 1/10} \geq 1 - 1/8$.
	To this aim we proceed as in \cref{eq:aida1} by using $F \leq 0$, $F = 0$ outside $B_{\sqrt{2}R}$ and \cref{poincare ball}
	\begin{align*}
		\int_{\R^{2m}} e^{2\eps^{-2}F }  \,d \gamma &= 1 - \int_{B_{\sqrt{2} R}} (1-e^{2\eps^{-2}F } )  \,d \gamma \\
		&\geq 1-   \left(\frac{4}{\eps^{2}\pi}\right)^m  \int_{B_{\sqrt{2} R}}e^{-4\eps^{-2}|W|^2} \, d W\\
		&\geq 1-\left(\frac{4}{\eps^{2}\pi}\right)^m|B_{\sqrt{2} R}|\\
		&\geq 1 - \frac{1}{\sqrt{2\pi m}}  \left( \frac{8 e R^2}{\eps^2m} \right)^m.
	\end{align*}
	By the assumptions $m \geq 24 R^2 \eps^{-2}$, $\eps \leq 1$ and $R \geq 10$ we have
	\begin{align*}
		\frac{1}{\sqrt{2\pi m}}  \left( \frac{8 e R^2}{\eps^2m}  \right)^m \leq \frac{1}{10}.
	\end{align*}
	Hence, \cref{poincare st5} follows from the two above inequalities and we have proved \cref{poincare st4}.

	Now using \cref{poincare st4,poincare st5,poincare st3} we see that \cref{eq:lemaida1,eq:lemaida2,eq:lemaida3} are satisfied and \cref{lem:aida} implies that the measure $ e^{2G_\eps }  \,d \gamma $, and thus also $ d \gamma_F = e^{2\eps^{-2}F} d\gamma$,   satisfies  the Poincar\'e inequality \cref{poincare} with constant (recall $\beta = \frac14$ and  $C_{LS} = \frac{\eps^2}{8}$)
	\begin{align*}
		C_{P,\eps}  \leq \frac{64 \cdot \frac{\eps^2}{8}(1+ \frac14)(1 - \frac18)}{ (\frac14)^2} = 140 \,  \eps^2.
	\end{align*}
	Setting $C_P = \eps^{-2} C_{P,\eps}$, the above inequality yields the second result.
\end{proof}

\cref{thm1}  follows from \cref{poincare thm1} by a standard argument which  we sketch for the convenience of the reader.
\begin{proof}[\textbf{Proof of \cref{thm1}}]
    Let us write $ \rho(W,t)=  u(W,t)e^{-\Phi_R(W)/\eps^2}$ in which case  for $\rho_\infty(W)=  e^{-\Phi_R(W)/\eps^2}$ we have
    \[
        \int_{\R^{2m}} (\rho(W,t)-\rho_\infty)^2 e^{\Phi_R/\eps^2}\, dW = \int_{\R^{2m}} (u(W,t)-1)^2 e^{-\Phi_R/\eps^2}\, dW.
    \]
    Recall that $u$ solves the equation \cref{FP3} and hence
    \begin{align} \label{eq:ddt}
        \notag \frac{d}{dt} \frac12 \int_{\R^{2m}} (u-1)^2 e^{-\Phi_R/\eps^2}\, dW &=  \int_{\R^{2m}} (u-1)\frac{\partial  u}{\partial t}  e^{-\Phi_R/\eps^2}\, dW\\
        \notag &= \int_{\R^{2m}}  (u-1)(\eps^2 \Delta u - \nabla \Phi_R \cdot \nabla u)  e^{-\Phi_R/\eps^2}\, dW\\
        \notag &= \eps^2 \int_{\R^{2m}}  (u-1) \, \nabla \cdot (\nabla u \,  e^{-\Phi_R/\eps^2})\, dW\\
        &= - \eps^2 \int_{\R^{2m}} | \nabla u|^2  e^{-\Phi_R/\eps^2} \, dW,
    \end{align}
    where the last step follows from the self-adjointness of \cref{FP3} in $L^2(e^{-\Phi_R/\eps^2} dW )$, see for instance \cite{Gross,Royer}.
    By \cref{poincare thm1} we have
    \begin{align} \label{eq:poincare}
        \int_{\R^{2m}} (u-1)^2 e^{-\Phi_R/\eps^2}\, dW \leq C_P \eps^2 \int_{\R^{2m}} | \nabla u|^2  e^{-\Phi_R/\eps^2} \, dW,
    \end{align}
    where the constant $C_P$ is explicitly given in \cref{poincare thm1}. Combining \cref{eq:ddt,eq:poincare} we obtain
    \[
        \frac{d}{dt}  \int_{\R^{2m}} (u-1)^2 e^{-\Phi_R/\eps^2}\, dW \leq - \frac{2}{C_P} \int_{\R^{2m}} (u-1)^2 e^{-\Phi_R/\eps^2}\, dW,
    \]
    which in turn implies
    \[
        \int_{\R^{2m}} (u(W,t)-1)^2 e^{-\Phi_R/\eps^2}\, dW \leq e^{-(2/C_P)\, t }\int_{\R^{2m}} (u_0(W)-1)^2 e^{-\Phi_R/\eps^2}\, dW .
    \]
    The above statement implies  the result.
\end{proof}

\section{Closure of \texorpdfstring{$\Hma$}{H\{m,a\}} in \texorpdfstring{$L^2$}{L2}}
\label{sec:closure}

The rest of the paper is devoted to study the Questions \cref{q:1}  and \cref{q:2}. We begin by  studying the closure of the space $\Hma$,  defined in \cref{def:NN}, with respect to the $L^2(\S^1)$-norm. We denote this closure by $\cHma$ and it contains the functions which we obtain as a limit of sequence of functions $(f_{W_k})_k$ in $\Hma$, when the weights diverge  to infinity  $|W_k| \to \infty$. We also study the structure of the symmetric part of functions in $\cHma$. 

We recall that a function $f_{W,a} = f_W \in \Hma$ can be written as
\begin{align}\label{eq:def_f3}
	f_{W}(x) = \frac{1}{\sqrt{m}} \sum_{i \in I} a_i \sigma(w_i \cdot x),
\end{align}
where $I = \{1, \dots, m \}$. As we are interested in the closure when the weight space $\R^{2m}$ is fixed,  we may  disregard the multiplicative factor $1/\sqrt{m}$.

\begin{remark}
	For further reference, the closure of $\Hma$ with respect to the $L^2(\S^1)$ norm is denoted as $\cHma$.
\end{remark}

The main result of this section is the following characterization of the space $\cHma$.
\begin{theorem} \label{thm4}
	A function $g : \S^1 \to \R$  belongs to the space $\cHma$ if and only if it is of the form
	\[
	g(x) = \sum_{i\in J} \I\{\hat w_i \cdot x \geq 0\} (v_i \cdot x) + \sum_{i \in K} a_i \sigma(w_i \cdot x),
	\]
	where $\hat w_i$ are unit vectors, the set of indexes $J,K \subset I$ are disjoint and $|J| \leq \underbar{m}$. Here $\underbar{m}$ is defined in \cref{def:underbar}.
\end{theorem}

\begin{remark} \label{rem:thm4}
	In order for a limit to exist in $L^2(\S^1)$ when certain nodes $|w_i| \to \infty$ as $k \to \infty$, some terms in the sum in \cref{eq:def_f3} needs to cancel. This cancellation is what gives rise to the index set $J$ and the remainder of that cancellation   gives rise to new functions of type $\I\{\hat w_i \cdot x \geq 0\} (v_i \cdot x)$. Note that these include the original functions $\sigma(w_i \cdot x)$ by choosing $v = w_i$.
\end{remark}
 
\cref{thm4} implies that every $g \in \cHma$ is a BV-function. In fact,  $g$ is linear on arcs which we call  \emph{sectors} and define as follows.
\begin{definition} \label{def:sector}
	Given an index set $I$ and a set of directions $\{\hat w_i:\hat  w_i \in \R^2\}$ we denote  
	\begin{align*}
		S_i = \{x \in \S^1: \hat w_i \cdot x \geq 0\}.
	\end{align*}
	Moreover, we define \emph{sectors} $\mathcal{S}$ as the connected components of $\S^1$ where the function 
	\[
	    x \mapsto \sum_{i \in I} \I\{ \hat w_i \cdot x \geq 0\}
	\]
	is constant.  In particular, the sectors are contained in the $\sigma$-algebra generated by  $S_i$.
\end{definition}

We split the proof of \cref{thm4} into two part in \cref{lem:limit1,lem:limit2}.  We begin with \cref{lem:limit1} which  identifies the closure.

\begin{lemma}
	\label{lem:limit1}
	Assume that $|W_k| \to \infty$ and $\|f_{W_k}\|_2 \leq C$. Then there is $g \in L^2(\S^1)$ and a sub-sequence of $W_k$ such that $f_{W_k} \to g$ in $L^2(\S^1)$. Moreover,  there exist disjoint sets of indexes $J, K \subset I$ such that $|J| \leq \underbar m$ and  $g$ can be written as
	\begin{align*}
		g(x) = \sum_{i \in J} \I\{\hat w_i \cdot x \geq 0\} (v_i \cdot x) + \sum_{i \in K} a_i \sigma(w_i \cdot x)
	\end{align*}
	for $w_i \in \R^2$, $i \in K$ and $\hat w_i,v_i \in \R^2$ for $i \in J$ such that $|\hat w_i| = 1$.
\end{lemma}

\begin{proof}
	Let us write
	\[
		f_{W_k}(x) = \sum_{i \in I} a_i \sigma(w_{i,k} \cdot x)
	\]
	and  fix $i$. If $\lim \inf_{k \to \infty} |w_{i,k}| = 0$ then clearly we may ignore these indexes. If  $\lim \sup_{k \to \infty} |w_{i,k}| <\infty$ we set $i \in K$ and define 
	\[
	    w_i := \lim_{k \to \infty} w_{i,k},
	\]
	which exists when we pass to a converging sub-sequence. If neither of these cases happen we set  $i \in I_\infty$. Then for 
	\[
	    f_{k, K} (x) :=\sum_{i\in I \setminus I_\infty}  a_i \sigma(w_{i,k} \cdot x) 
	\]
	we have that
	\[
        f_{k, K} (x)  \to \sum_{i\in K}  a_i \sigma(w_{i} \cdot x).
	\]
	We thus need to prove the convergence of 
	\[
	    f_{k,J} : = \sum_{i\in I_\infty}  a_i \sigma(w_{i,k} \cdot x).
	\]
	
    Let us fix $i \in I_\infty$ and denote $\hat w_{i,k} = \frac{w_{i,k}}{|w_{i,k}|}$. By passing to a sub-sequence we may we define  
	\[
		\hat w_i := \lim_{k \to \infty} \hat w_{i,k}  
	\]
	and assume that $\lim_{k \to \infty} |w_{i,k}| = \infty$.  We claim that for each $i \in I_\infty$ there is $j \in I_\infty$ such that $ \hat w_i = \hat w_j$ and $a_i = -a_j$. Indeed, if this is not the case then the limit $\lim_{k \to \infty} \|f_{W_k}\|_2 $ would be unbounded. Therefore we in particular deduce that the number of unique directions of $\hat w_i$ for $i \in I_\infty$ is at most $\underbar{m}$. The indices of these directions  define the index set  $J \subset I_\infty$.

	Let us write 
	\[
	    f_{k,J}(x) = \sum_{i \in I_\infty} a_i \I\{ \hat w_{i,k} \cdot x \geq 0\} (w_{i,k} \cdot x) .
	\]
	Let us next fix a small $\eps>0$ and define
	\begin{align*}
		\Gamma_\eps := \{ x \in \S^1 : |\hat w_i \cdot x| \leq \eps  \,\, \text{for some }\, i \in J  \}.
	\end{align*}
	Since $\lim_{k \to \infty} \hat w_{i,k} = \hat w_i$, then for $i \in I_\infty$ and large $k$ we have
	\begin{equation*}
		\{ x \in \S^1 \setminus \Gamma_\eps : (\hat w_{i,k} \cdot x) \geq 0 \}   \subset    \{ x \in \S^1 : (\hat w_i \cdot x) \geq 0 \}.
	\end{equation*}
	Therefore, for all $x \in \S^1 \setminus \Gamma_\eps $ we may write
	\[
		f_{k,J}(x) = \sum_{i \in I_\infty} a_i \I\{ \hat w_{i} \cdot x \geq 0 \} (w_{i,k} \cdot x).
	\]
	To continue we first prove that
	\begin{align} \label{eq:bounded_0}
		\sup_{x \in \S^1} |\sum_{i \in I_\infty}  \I\{ \hat w_{i} \cdot x \geq 0 \} a_i w_{i,k}| \leq C'
	\end{align}
	for all $k$. Then we may use compactness to obtain a convergent sub-sequence.

	We note that for an arbitrary arc $S \subset \S^1$ and a vector $v \in \R^2$ we have the following  bound
	\begin{align} \label{sectorbound}
		\int_S (v \cdot x)^2 dx \geq c |v|^2 |S|^3,
	\end{align}
	for some  $c>0$. We leave this easy bound for the reader to check. To prove \cref{eq:bounded_0}, we let $\mathcal{S}$ be all the sectors associated with the directions $\{ \hat w_i \}_{i \in J}$ given by \cref{def:sector}. Fix $S \in \mathcal{S}$ and let $I_S \subset I_\infty$ be the indices $i \in I_S$ for which $\hat w_{i} \cdot x > 0$ in that sector. From \cref{sectorbound} and from
	\begin{align*}
		 \frac{1}{2 \pi}\int_{\S^1 \setminus \Gamma_\eps}    \left(\sum_{i \in I_\infty}  a_i \I\{ \hat w_{i} \cdot x \geq 0 \} (w_{i,k} \cdot x) \right)^2 \, dx \leq \| f_{k,J}\|_{L^2}^2  \leq C
	\end{align*}
	we deduce
	\begin{align*}
		c |S\setminus \Gamma_\eps|^3 \left | \sum_{i \in I_S} a_i w_{i,k}\right |^2 \leq \int_{S \setminus \Gamma_\eps}  \left(\sum_{i \in I_S} a_i w_{i,k} \cdot x\right)^2 \, dx \leq C.
	\end{align*}
	Note that $|S \setminus \Gamma_\eps| \geq |S| -2m\eps$. Therefore since the sector $S$ was arbitrary we obtain \cref{eq:bounded_0} when $\eps>0$ is chosen small enough.

	The bound \cref{eq:bounded_0} implies that at each sector $S$ and for the associated index set $I_S$ there is a converging sub-sequence $\sum_{i \in I_S} a_i w_{i,k} \to \tilde v_S$ as $k \to \infty$. Therefore, we deduce that $f_{k,J}$ converges to $\tilde g$ in $L^2(\S^1)$ and the limit function can be written as 
	\begin{align*}
		\tilde g(x)= \sum_{S} \I\{x \in S\} (\tilde v_S \cdot x).
	\end{align*}
	Recall that the sectors are by definition the arcs where 
	\[
	    x \mapsto \sum_{i \in J} \I\{\hat w_i \cdot x \geq 0\}
	\]
	is constant. Therefore it is clear that we may write $g_1$ as
	\[
	    g_1(x) = \sum_{i \in J} \I\{\hat w_i \cdot x \geq 0\} (v_i \cdot x)
	\]
	for  $v_i \in \R^2$. This concludes the proof. 
\end{proof}

Next we prove the other implication of the statement of \cref{thm4}, i.e., that every function of the type stated in \cref{thm4} can be obtained as a limit of functions in $\Hma$.

\begin{lemma} \label{lem:limit2}
	Assume that a function $g \in L^2(\S^1)$ is of the form 
	\begin{align*}
		g(x) = \sum_{i\in J} \I\{\hat w_i \cdot x \geq 0\} (v_i \cdot x) + \sum_{i \in K} a_i \sigma(w_i \cdot x),
	\end{align*}
	where $\hat w_i$ are unit vectors and the set of indexes $J,K \subset I$ are disjoint and $|J| \leq \underbar{m}$. Then $g$ is in $\cHma$.
\end{lemma}

\begin{proof}
	We recall that by \cref{rem:reorder}, we assume that the first $2 \underbar m$ coefficients $a_i$ are alternating, i.e.,  we assume that   $a_1 = 1 = -a_2, \dots$. We may thus construct the following function
	\[
		f_{W_h}(x) = \frac{1}{h} \left( \sum_{i = 1}^{2|J|} \sigma((\hat w_i + hv_i) \cdot x) -  \sigma(\hat w_i \cdot x) \right) + \sum_{i \in K} a_i \sigma(w_i \cdot x)
	\]
	in the space of neural networks $\Hma$. From the above construction it follows that $f_{W_h}(x)$ is uniformly bounded and $\lim_{h \to 0} f_{W_h}(x) = g(x)$ a.e $x$.
\end{proof}

Recall that $\Hma \subset \Lal$, where $\Lal$ is defined in \cref{L2al}. Therefore it is trivial that also $\cHma \subset \Lal$. This means that we may write any function $g \in \cHma$ as $g = g^s + l$, where $g^s$ is symmetric and $l$ is linear. In the next lemma we characterize the symmetric and the linear part of $g$.

\begin{proposition} \label{lem:symasym}
	Assume $g \in \cHma$ is written as
	\begin{align*}
		g(x) = \sum_{i\in I} \I\{w_i \cdot x \geq 0\} (v_i \cdot x).
	\end{align*}
	Then we may split $g$ into a symmetric $g^s(x)$ and an anti-symmetric $g^a(x)$ part such that
	\begin{align*}
		g(x) = g^s(x) + g^a(x),
	\end{align*}
	where $g^a$ is linear. Furthermore, the symmetric part can be written as
	\begin{align*}
		g^s(x) = \frac12\sum_{i \in I} \sgn (w_i \cdot x) (v_i \cdot x)
	\end{align*}
	and the anti-symmetric part $g^a(x)$ can be written as
	\begin{align*}
		g^a(x) = \frac12 \left (\sum_{i \in I} v_i \right ) \cdot x.
	\end{align*}
\end{proposition}
\begin{proof}
	We begin by writing the symmetric part of $g$ as
	\begin{align*}
		g^s (x) = \frac{g(x)+g(-x)}{2}.
	\end{align*}
	Applying this to $g$ we obtain
	\begin{align*}
		2g^s(x) &= \sum_{i \in I} \I\{w_i \cdot x \geq 0\} (v_i \cdot x) - \sum_{i  \in I} \I\{-w_i \cdot x \geq 0\} (v_i \cdot x) \\
		&= \sum_{i \in I} \sgn (w_i \cdot x) (v_i \cdot x).
	\end{align*}
	On the other hand, we may write the anti-symmetric part as 
	\begin{align*}
		g^a(x) = \frac{g(x)-g(-x)}{2},
	\end{align*}
	which then becomes
	\begin{align*}
		2g^a (x) &= \sum_{i \in I} \I\{w_i \cdot x \geq 0\} (v_i \cdot x) + \sum_{i  \in I} \I\{-w_i \cdot x \geq 0\} (v_i \cdot x) \\
		&= \sum_{i\in I} \left (\I\{w_i \cdot x \geq 0\}+ \sum_{i\in I} \I\{-w_i \cdot x \geq 0\} \right ) (v_i \cdot x )\\
		&= \left (\sum_{i \in I} v_i \right ) \cdot x.
	\end{align*}
\end{proof}

We conclude this section by noticing that if $y$ is anti-symmetric, i.e., $y(-x) = -y(x)$, then the solution of the minimization problem \cref{eq:l2approx} is trivial. Indeed,  if $y$ is linear, $y(x) =  w_0 \cdot x$, then we may easily write it in terms of $f_{W,a}$. Indeed, assuming there are $i,j$ such that  $a_i = -a_j=1$, we choose $w_i = - w_j = \sqrt{m} \, w_0$ and for all other indexes $w_k = 0$. For this choice we clearly have
\begin{align*}
	f_{W,a}(x) = \sigma(w_i \cdot x) -\sigma(w_j \cdot x) = \I_{w_0 \cdot x \geq 0} (w_0 \cdot x)  - \I_{w_0 \cdot x \leq 0} (- w_0 \cdot x)  = w_0 \cdot x.
\end{align*}
In case $a_i = 1$ for all $i$, then $f_{W,a}(x) = \sigma(\alpha \cdot x) $ is clearly the best approximation of $y$. On the other hand, it turns out that if $y$ is anti-symmetric and orthogonal to the set of linear functions, then the best approximation is given by the zero function. We state this in the next lemma.

\begin{lemma}\label{anti-symmetric}
	Assume $y \in L^2(\S^1)$, let us write it as  $y  =y_1 + y_2$, where $y_1 \in \Lal$ and $y_2 \in \Lal^\perp$. Then, for any $f \in \Lal$ we have
	\begin{align*}
		\|f - y\|_2^2 = \|f - y_1\|_2^2 + \|y_2\|_2^2.
	\end{align*}
	In particular, if $y \in \Lal^\perp$ then the best approximation is given by the zero function, i.e.,
	\begin{align*}
		\inf_{W \in \R^{2m}} \|f_W- y\|_2^2  = \| y \|_2^2.
	\end{align*}
\end{lemma}
\begin{proof}
    The claim follows from the orthogonality of the vector spaces $\Lal$ and $\Lal^\perp$ and the second claim follows from the fact that $\Hma \subset \Lal$.
\end{proof}

By \cref{lem:symasym,anti-symmetric} one may also consider the approximation problem only in the space of symmetric functions. Indeed, by \cref{lem:symasym} this is equivalent to constraining the nodes $v_i$ in the representation of  \cref{lem:symasym} to satisfy 
\[
    \sum_{i \in I} v_i = 0.
\]

\section{Uniform approximation theorem}
\label{sec:thm3}
In this section we prove \cref{thm3} which is the uniform approximation theorem for BV functions which are in the space $\Lal$. The assumption that $y \in \Lal$ is natural by \cref{lem:symasym,anti-symmetric} and for  the assumption  $y \in BV(\S^1)$ we refer to \cref{rem:local1}.

The question of universal approximation is fundamental in the theory of neural networks and let us therefore discuss a classical result. Let us consider a more general space of single hidden layer neural networks $\mathcal{H}(m)$, which can be written as
\begin{align*}
	f_{W,a}(x) = \sum_{i=1}^m a_i \sigma(w_i \cdot x + b_i) : \Omega \subset \R^n \to \R
\end{align*}
where $\sigma$ is a general activation function and $b_i \in \R$. It is shown, e.g., in \cite{LLPS} that this space of functions satisfy the so called universal approximation theorem, which we state below. 

\begin{theorem*}[Universal Approximation Theorem \cite{LLPS}] \label{thm:universal}
	Let $\sigma$ be an activation function. Set
	\begin{align*}
		\Sigma_n = \text{span} \{\sigma(w \cdot x + \theta): \quad w \in \R^n,\, \theta \in \R\}
	\end{align*}
	then $\Sigma_n$ is dense in $L^p(\mu)$ for any measure $\mu$ iff $\sigma$ is non-polynomial and non-constant.
\end{theorem*}
The universal approximation theorem gives us qualitative information about the class of functions that a neural network can represent. However, it does not give information about the  quantitative approximation error when the number of nodes is finite.

The key idea of the proof of the above theorem in \cite{LLPS} is that as the size of the weights diverge to infinity we may construct difference quotients w.r.t. the weights (assuming $\sigma \in C^\infty$). This gives rise to polynomials and the result follows from the Weierstrass approximation theorem. As we discussed in \cref{rem:thm4}, the effect of cancellation is the key observation also in our setting.

We may also study this problem from the point of view of function spaces. By our earlier results we know that the space $\cHma$ for every $m$ is contained in $BV(\S^1) \cap \Lal$, where $\Lal$  is defined in \cref{L2al}. We would like to know if the space  $\cHma$ is dense in $BV(\S^1) \cap \Lal$ when we let $m$ go to infinity. We solve this problem by constructing an explicit family of simple functions in $\cHma$ which approximate any symmetric (antipodally) BV-function well. 

We begin by showing that in the four node case, we may approximate narrow symmetric step functions well.
\begin{lemma} \label{step function}
    Assume that $m=4$ and $\underbar m=2$, where $\underbar m$ is defined in \cref{def:underbar}. Let $y$ be a symmetric step function, i.e., in polar coordinates $y(\theta) = c\I_{(\theta_1,\theta_2)} + c\I_{(\theta_1 +\pi,\theta_2+\pi)}$, where  $0< \theta_2-\theta_1 < \pi$ and $c \in \R$.  There is a function $g \in \cHma$ such that in polar coordinates $g(\theta) = 0$ for all $x \notin [\theta_1,\theta_2] \cup [\theta_1 +\pi,\theta_2+\pi]$, i.e., $\text{spt} (g) \subset ([\theta_1,\theta_2] \cup [\theta_1 +\pi,\theta_2+\pi])$ and 
    \[
        \|y - g\|_{2}^2 \leq \frac{c^2}{1000} (\theta_2- \theta_1)^5.
    \]
\end{lemma}

\begin{proof}
By rotation and scaling we may assume that $\theta_1 = - \theta_0 $, $\theta_2 = \theta_0$ with $\theta_0 < \pi/2$ and $c=1$. Now, using \cref{thm4} we see that there is a function $g \in \cHma$ such that $g(x) = \I\{\hat w_1\cdot x \geq 0\} x_1 - \I\{\hat w_2\cdot x \geq 0 \} x_1$ with weights $v_1 = -v_2 = e_1$,  
\begin{align*}
    \hat w_1 = 
    \begin{bmatrix}  
        \sin \theta_0 \\  
        \cos \theta_0 
    \end{bmatrix}  \qquad \text{and} \qquad 
    \hat w_2 = 
    \begin{bmatrix} 
        - \sin \theta_0 \\  
        \cos \theta_0 
    \end{bmatrix} 
\end{align*}
and $J= \{ 1,2\}$. We may write $g$ in polar  coordinates  as
\begin{align*}
g(\theta) = \I_{(-\theta_0,\theta_0)}\cos \theta  -  \I_{(-\theta_0 +\pi,\theta_0+\pi)}\cos \theta.
\end{align*}
Let us check that $g$ satisfies the required  conditions. 

First, it is clear that $\text{spt} (g) \subset ([-\theta_0,\theta_0] \cup [-\theta_0 +\pi,\theta_0+\pi])$.  Second, by using $|1 - \cos \theta|\leq \theta^2/2$ for $|\theta| < \pi/2$ we estimate 
\[
\begin{split}
    \|y - g\|_{2}^2 &= \frac{1}{2\pi}\left( \int_{(-\theta_0,\theta_0)} (1 - \cos \theta)^2 \, d \theta + \int_{(-\theta_0+\pi,\theta_0+\pi)} (1 + \cos \theta)^2 \, d \theta \right)\\
    &=  \frac{2}{\pi} \int_0^{\theta_0} (1 - \cos \theta)^2 \, d \theta \leq \frac{1}{2\pi} \int_0^{\theta_0} \theta^4 \, d \theta\\
    &= \frac{1}{10\pi} \theta_0^5 
    \leq \frac{1}{1000} (2 \theta_0)^5,
\end{split}
\]
which completes the proof.
\end{proof}

We are now ready to prove \cref{thm3}. We restate it for the convenience of the reader and recall that the condition \cref{more than symmetric} means that $y \in \Lal$. Recall also the definition of $\Hma$ in \cref{def:NN}.
\bvtheorem*

\begin{proof}[\textbf{Proof of \cref{thm3}}]
Let us begin by approximating the symmetric part $y^s$ with  symmetric simple functions. We claim that given a number $N \in \N$  we may find a function, denote it  $v_N$, which is a sum of $N$ many symmetric step functions  such that 
\begin{equation} \label{thm3_1}
    \|y^s -v_N\|_{L^2(\S^1)}^2 \leq  \frac{\pi^2 \|y\|_{BV}^2  }{N},
\end{equation}
furthermore, $v_N$ can be written as 
\begin{equation*}
    v_N(\theta) = \sum_{k=1}^N c_k \I_{I_k} \text{ for } \theta \in [0,\pi), |c_k| \leq \|y^s\|_{L^\infty}, \text{ and } |I_k| = \frac{\pi}{N}.
\end{equation*}
We remark that the function $v_N$ above is not in the function spaces $\Hma$ or $\cHma$. We postpone the proof of \cref{thm3_1} at the end.

The proof now goes as follows. First, recall that  by \cref{rem:reorder}  $\underbar m$ is the number of negative coefficients $|\{ i \in I : a_i = -1 \}|$ and that the first $2 \underbar m$ coefficients are alternating $1=a_1 = -a_2 =\dots$. We may also clearly assume that $\underbar m \geq 1$. By the assumption $y \in \Lal$, $y$ is of the form $y= y^s + l$ and thus we use the first two nodes to construct the linear function $l$. If $\underbar m \leq 2$, we cannot, in general, do better than this and thus we set all other weights to zero. On the other hand if $\underbar m \geq 3$ then we may  approximate the simple function $v_N$ from \cref{thm3_1} with the remaining nodes. 

Let us thus assume that $l$ is the linear part of $y$. We may choose $g_l \in \Hma$ as (recall that $a_1 = 1$ and $a_2 = -1$)
\[
    g_l(x) = a_1 \sigma (w_0 \cdot x) + a_2 \sigma (-w_0 \cdot x) = l(x).
\]
If $\underbar m \leq 2$ then we choose all the other weights to be zero and have $\|y - g_l\|_2^2 = \|y^s\|_2^2 \leq   2\|y\|_{BV}^2/\underbar m$ and the result follows.

Assume that  $\underbar m \geq 3$ and denote $N := \lfloor \frac{\underbar m -1}{2} \rfloor$ and observe that $N \geq \underbar m /3$. Let $v_N$ be the function given by \eqref{thm3_1}.  Now the coefficients $a_i$ are alternating for $i =3, \dots, 2 \underbar m$ and by the choice of $N$ we have $2N \leq (\underbar m -1)$. Therefore we may use \cref{step function} for all $k= 1, \dots, N$ and find  function $g_{s} \in \cHma$ of the form 
\[
    g_{s}(x) = \sum_{i=3}^{2 \underbar m} \I\{\hat w_i \cdot x \geq 0\} (v_i \cdot x)
\]
such that (recall that $|I_k| = \pi/N$)
\begin{align*}
    \|v_N - g_s\|_{2}^2 &\leq \frac{1}{1000}\|y^s\|_{\infty}^2 \sum_{k=1}^N |I_k|^5  =  \frac{\pi^5}{1000} \|y^s\|_{\infty}^2 \frac{1}{N^{4}} \\
    &\leq \frac{3 \pi^5\|y\|_{BV}^2}{1000 \underbar m} \leq \frac{\|y\|_{BV}^2}{\underbar m},
\end{align*}
where in the second inequality we have used $N \geq \underbar m/3 \geq 1$. We choose $g \in \cHma$ as $g = g_l +g_s$ (note $g_l = l$)  and deduce by \eqref{thm3_1}  and $N \geq \underbar m /3$ that 
\begin{align*}
    \inf_{W\in \R^{2m}} \| f_W - y\|_2^2 &\leq  \|g-y\|_{2}^2 = \|y_s - g_s\|_{2}^2 \\
    &\leq 2\|y_s - v_N\|_{2}^2 + 2 \|v_N - g_s\|_{2}^2\\
    &\leq  \frac{2\pi^2 \|y\|_{BV}^2  }{N} + \frac{2 \|y\|_{BV}^2}{\underbar m} \leq \frac{62 \|y\|_{BV}^2}{\underbar m}.
\end{align*}
Hence, we conclude the proof once we have established \cref{thm3_1}.

To prove \cref{thm3_1} we divide  $[0,\pi)$ into $N$-many intervals $I_1, I_2, \dots, I_N$, $I_k = [\theta_k, \theta_{k+1})$, with equal length $|I_k| = \frac{\pi}{N}$. First we note that we consider a right continuous representative of $y$, and define the simple function $v_N: [0,2\pi) \to \R$ as follows
\begin{align*}
    v_N(\theta) = 
    \begin{cases}
        y^s(\theta_k), & \text{if } \theta \in I_k \text{ or } \theta-\pi \in I_k \\
        0, & \text{otherwise,}
    \end{cases}
\end{align*}
where we represent the function $y^s$ in polar coordinates. Since $y^s$ is a BV-function it has a derivative which is a Radon measure that we denote  by $\mu_{y^s}$. Recall that by the fundamental theorem for BV-functions and the right continuity of $y$, we have  (see \cite[p.139]{AFP})
\begin{align*}
    |y^s(\theta) - y^s(\theta_k)| = |\int_{\theta_k}^{\theta} d \mu_{y^s}| \leq |\mu_{y^s}|(I_k)  \qquad \text{for a.e. } \, \theta \in I_k,
\end{align*}
where $|\mu_{y^s}|$ denotes the total variation of $\mu_{y^s}$. Therefore, by symmetry 
\begin{align*}
    |\mu_{y^s}|(I_k) \leq  |\mu_{y^s}|([0,\pi)) \leq \pi \|y^s\|_{BV},
\end{align*}
and thus
\begin{align*}
    |y^s(\theta) - y^s(\theta_k)|^2 \leq (|\mu_{y^s}|(I_k))^2 \leq \pi \|y^s\|_{BV} |\mu_{y^s}|(I_k),
\end{align*}
for all  $\theta \in I_k$. We use the above inequality to estimate (by symmetry)
\begin{align*}
    \|y^s -v_N\|_{L^2(\S^1)}^2 &=  \frac{2}{2\pi}\sum_{k=1}^N \int_{I_k} |y(\theta) - y(\theta_k) |^2 \,d \theta \leq \|y^s\|_{BV}  \sum_{k=1}^N \overbrace{|I_k|}^{=\pi /N} |\mu_{y^s}|(I_k) \\
    &=\frac{ \pi \|y^s\|_{BV}  }{N} \sum_{k=1}^N |\mu_{y^s}|(I_k) =\frac{\pi \|y^s\|_{BV}  }{N} |\mu_{y^s}|([0,\pi)) \\
    &\leq  \frac{\pi ^2 \|y^s\|_{BV}^2  }{N}. 
\end{align*}
The inequality \cref{thm3_1} then follows from $\|y^s\|_{BV}^2  \leq \|y\|_{BV}^2 $. This completes the proof of the theorem.
\end{proof}

\section{The Localization theorem}
\label{sec:thm2}

The goal of this  section  is to prove \cref{thm2}. We begin with the observation that by \cref{thm4} and the identity $\sigma(w_i \cdot x) = \I \{\hat w_i \cdot x \geq 0 \} (w_i \cdot x)$, we may write every function $g \in \cHma$ as
\begin{align} \label{eq:g_repr}
    g(x) =  \sum_{i \in I} \I\{\hat w_i \cdot x \geq 0\}(v_i \cdot x),
	\end{align}
where $\hat w_i \neq \hat w_j$ with $i \neq j$. In particular, every function $g \in \cHma$ is locally linear, i.e., it is linear on arcs which we called sectors. Where the sectors are defined as the components where the function
\[
    x \mapsto \sum_{i \in I} \I\{\hat w_i \cdot x \geq 0\}
\]
is constant (see \cref{def:sector}). We recall also that we denote $S_i=  \{x \in \S^1 : \hat w_i \cdot x \geq 0 \}$ and that the the sectors, denoted by $\mathcal{S}$, belong to the $\sigma$-algebra generated by $S_i$.   

Since we have the characterization of the closure of $\Hma$ by \cref{thm4}, we may actually find the function which realizes the infimum in \cref{eq:l2approx}, i.e., there is $g \in \cHma$ which is of the form \cref{eq:g_repr} such that 
\begin{align} \label{infimum-again}
	\inf_{f_W \in \Hma}\|f_W - y\|_2^2 = \|g - y\|_2^2.
\end{align}

In the next lemma we show that the minimizer $g$ solves a kind of Euler-Lagrange equation, which will be useful later when we study the regularity properties of $g$. As we mentioned in the introduction we will turn these regularity estimates into information on the size of the vectors $v_i$ in \cref{eq:g_repr}. We stress that unlike the previous section, here we do not assume the function $y$ to be in the space $\Lal$ defined in \cref{L2al}. 

\begin{lemma} 
	Let $y \in L^2(\S^1)$. Then there exists a function $g \in \cHma$ of the form \cref{eq:g_repr} given by  \cref{lem:symasym} which  realizes the infimum  \cref{infimum-again}. Furthermore, if $g^s$ and $y^s$ are the symmetric part of $g$ and $y$ respectively, then for each sector $S \in \mathcal{S}$ associated with $\{\hat w_i\}_{i \in I}$ in  \cref{eq:g_repr} (see \cref{def:sector}) we have
	\begin{equation}
		\label{from euler}
	 	\int_{S}  g^s(x) \, x \, dx = \int_{S} y^s(x) \, x \, dx.
	\end{equation}
\end{lemma}
\begin{proof}
    From \cref{thm4} we deduce that there is a function $g \in \cHma$ which realizes the infimum in \cref{infimum-again} and we may write it as 
	\begin{align*}
		g(x)=  \sum_{i \in J} \I\{\hat w_i \cdot x \geq 0\} (v_i \cdot x) + \sum_{i \in K} a_i \sigma(w_i \cdot x),
	\end{align*}
    where $\hat w_i \neq \hat w_j$ for $i \neq j$. Let us fix $i, j \in I$, with $i \neq j$, let $S_i = \{ \hat w_i \cdot x \geq 0 \}$ and $S_j= \{ \hat w_j \cdot x \geq 0 \}$   and denote the arcs 
    \[
        S_{ij} = S_i \setminus S_j \qquad \text{and} \qquad S_{ji} = S_j \setminus S_i.
    \] 
    Let us assume that $i \in J$ and $j \in K$. (The other cases follow from a similar argument). As in the proof of \cref{anti-symmetric} we decompose  norm as follows
	\begin{align*} 
		\|g -y\|_2^2 &= \|g^s - y^s\|_2^2+\|g^a -y^a\|_2^2,
	\end{align*}
	where $g^s$ and $g^a$ are the symmetric and the anti-symmetric part of $g$, and  $y^s,y^a$ are the symmetric and the  anti-symmetric part of y. Let us further fix a vector $u \in \R^2$. Since $g$ is a minimizer of $\|f - y \|_{2}^2$ in $\cHma$ we may consider the following variation of $g$ 
    \[
    \begin{split}
        g_{t}(x)= g(x) + t \, \I\{\hat w_i \cdot x \geq 0 \}(u\cdot x) + a_j(\sigma( (w_j - a_j t u) \cdot x)) - \sigma( w_j \cdot x)),
    \end{split}
    \]
    for all $t \in \R$. Then, we deduce by \cref{lem:symasym} that for the anti-symmetric part of $(g_t)^a$ we have
    \[
        (g_t)^a(x) =\frac12 \left( \sum_{k \in J\setminus \{i\}} v_k + (v_i + t u) + \sum_{k \in K\setminus \{j\}} a_k w_k + a_j(w_j- a_j tu)  \right)\cdot x = g^a(x),
    \]
    for all $t$. Therefore, we deduce by the above discussion that the following holds
    \begin{align} \label{eq:gs_var}
		\|g_t -y\|_2^2 &= \|g_t^s - y^s\|_2^2+\|g^a -y^a\|_2^2.
	\end{align}
    Now clearly $g_t \in \cHma$ and by the minimality of $g$ together with \cref{eq:gs_var} we get
	\[
		\frac{d}{dt}\bigg{|}_{t =0} \|g_t^s -y^s\|_2^2 =0. 
	\]
	By recalling the choice of $g_t$ and the form of the symmetric part $g^s$ by \cref{lem:symasym} we obtain 
	\begin{align*} 
		\int_{\S^1} \left ( \sgn(w_i \cdot x \geq 0 )-\sgn (w_j \cdot x \geq 0 ) \right ) (g^s(x) -y^s(x)) (x \cdot u)\, dx = 0.
	\end{align*}
	Using $S_{ij} = S_i \setminus S_j$ and $S_{ji} = S_j \setminus S_i$, the above can be rewritten as
	\begin{align*}
		\int_{S_{ij}}  (g^s(x) -y^s(x)) \, (x \cdot u) \, dx =  \int_{S_{ji}}   (g^s(x) -y^s(x)) \, (x \cdot u) \, dx.
	\end{align*}
	Since $g^s$ and $y^s$ are symmetric, then the functions $x \mapsto x \, g^s(x)$ and $x \mapsto x \, y^s(x)$ are anti-symmetric (component-wise). Hence, we may simplify the above equality as
	\[
		\int_{S_{ij}} (g^s(x) - y^s(x)) \, (x \cdot u) \, dx =   0.
	\]
	Since  this  holds for all $u \in \R^2$ we have  
	\begin{align} \label{eq:eulersect0}
	 	\int_{S_{ij}}  g^s(x) \, x \, dx = \int_{S_{ij}} y^s(x) \, x \, dx .
		\end{align}
	Again by the fact that $x \mapsto x \, g^s(x)$ and $x \mapsto x \, y^s(x)$ are anti-symmetric, we deduce that 
	\[
	    \int_{\S^1}  g^s(x) \, x \, dx =  \int_{\S^1} y^s(x) \, x \, dx = 0.
	\]
	Therefore we have by \cref{eq:eulersect0}  that 
	\[
	    \int_{\S^1 \setminus S_{ij}}  g^s(x) \, x \, dx = \int_{\S^1 \setminus S_{ij}} y^s(x) \, x \, dx .
	\]
	The claim follows from this and from \cref{eq:eulersect0} by using the symmetry of $y^s$ and $g^s$ and the fact that every sector is in the $\sigma$-algebra generated by the sets $S_i$.
\end{proof}

Let us consider the function $g$ which realizes the infimum in \cref{infimum-again} and write it as \cref{eq:g_repr} with index set $I$ and directions $\hat w_i$. By  \cref{def:sector} both  $g$ and  $g^s$ are  linear on the sectors $S \in \mathcal{S}$ associated with the directions $\{ \hat w_i\}$. Let us fix such a sector $S$ and write the symmetric part $g^s$ on $S$ as 
\begin{equation} \label{tilde_v} 
    g^s(x) = \frac{1}{\sqrt{m}} (\tilde{v}_S\cdot x).
\end{equation}

In the next lemma  we use the Euler-Lagrange equation \cref{from euler} to prove that the minimizer of  \cref{infimum-again}  inherits the regularity bounds from  the function $y$ when  $y$ is $C^1$-regular.

\begin{lemma} \label{lem:eulersecbound}
	Assume that $y \in C^1(\S^1)$  and  that  $g \in \cHma$ realizes the infimum in \cref{infimum-again} and let us write it as in \cref{eq:g_repr}. Let $S$ be any sector as in \cref{def:sector}, and let $\hat w_i$ be a vector such one of the boundary points of $S$ is on $\{x: \hat w_i \cdot x = 0\}$. Then the  symmetric part $g^s$ is linear on $S$ and the associated   vector $\tilde v_S$ (see \cref{tilde_v})  satisfies
	\[
		\frac{1}{\sqrt{m}} |\tilde v_S\cdot  w_i^\perp| \leq (4 + \pi^2) \|y\|_{\infty}
	\]
	and
	\[
		\frac{1}{\sqrt{m}} |\tilde v_S| \leq 3 \pi^2  \min \big{\{} \|y\|_{C^1}, \frac{ \|y\|_{\infty}}{|S|} \big{\}},
	\]
	where    $ w_i^\perp$ is unit vector   orthogonal to $\hat w_i$.
\end{lemma}

\begin{proof}
	Let us fix a sector $S$ and choose the coordinates in $\R^2$ such that $S = \{ \tilde \sigma(\theta) \in \R^2 : - \theta_0 \leq \theta \leq \theta_0 \}$,  where
	\begin{align*}
		\tilde \sigma(\theta) =
		\begin{bmatrix}
					\cos \theta \\
					\sin \theta
		\end{bmatrix} .
	\end{align*}
	Then clearly  
	\begin{align} \label{eq:wperp}
	    w_i^\perp = 
	    \begin{bmatrix}
			\cos \theta_0  \\
			\pm \sin \theta_0
		\end{bmatrix}.
	\end{align}
	Denote also $v_1 := \frac{1}{\sqrt{m}} \tilde v_S \cdot e_1 $ and $v_2 := \frac{1}{\sqrt{m}} \tilde v_S \cdot e_2$. We need to estimate $|v_1|$ and $|v_2|$.

	We assume that the functions are written in polar coordinates and  calculate both sides  of \eqref{from euler}, which  reads as 
	\begin{equation}
	    \label{from_euler2}
	    \fint_{S} (\tilde{v}_S \cdot x)  \, x \, dx = \fint_{S} y^s(x) \, x \, dx.
	\end{equation}
	We estimate the first component of the  RHS of \eqref{from_euler2} simply  as
	\begin{align} \label{eq:ys_sector_bound}
		\big{|}\fint_{S} y^s(x) \, x_1 \, dx \big{|} \leq \|y^s\|_{\infty}
	\end{align}
	while the first component of the LHS of \cref{eq:ys_sector_bound} is
	\begin{align} \label{eq:vs_sector_bound}
		\notag \frac{1}{\sqrt{m}} \fint_{S}  (\tilde v_S \cdot x)\, x_1 \, dx  &= \fint_{-\theta_0}^{\theta_0} (v_1 \cos \theta + v_2 \sin \theta)\cos \theta \, d \theta \\
		&= v_1 \fint_{-\theta_0}^{\theta_0} \cos^2 \theta \, d \theta .
	\end{align}
	Since $\big{|} \fint_{-\theta_0}^{\theta_0} \cos^2 \theta \, d \theta \big{|} \geq \frac{1}{4}$ we deduce from \cref{eq:ys_sector_bound,eq:vs_sector_bound,from_euler2}  that
	\begin{equation}
	    \label{from_euler3}
		|v_1| \leq 4 \|y^s\|_{\infty} \leq 4 \|y\|_{\infty} .
	\end{equation}
	We continue by calculating the second component of the LSH of \cref{from_euler2} as
	\begin{equation} \label{lem: euler bound1}
		\begin{split}
			\fint_{S}  (\tilde v_S \cdot x)\, x_2 \, dx  &= \fint_{-\theta_0}^{\theta_0} (v_1 \cos \theta + v_2 \sin \theta)\sin \theta \, d \theta \\
			&= v_2 \fint_{-\theta_0}^{\theta_0} \sin^2 \theta \, d \theta = v_2 \frac{1}{2\theta_0}\left( \theta_0 - \frac12 \sin (2\theta_0) \right).
		\end{split}
	\end{equation}
	We estimate the second component of the RSH of \cref{from_euler2} first as
	\begin{align} \label{lem:euler_bound1+}
		\big{|}\fint_{S} y^s(x) \, x_2 \, dx \big{|} \leq \|y^s\|_{\infty}  \fint_{-\theta_0}^{\theta_0}  |\sin \theta| \, d \theta \leq \|y^s\|_{\infty}  \fint_{-\theta_0}^{\theta_0}  |\theta| \, d \theta \leq  \|y^s\|_{\infty} \frac{\theta_0}{2}.
	\end{align}
	For all $\theta_0 \in [0, \pi]$ we may bound
	\begin{equation} \label{lem: euler bound2}
		\theta_0 - \frac12 \sin (2\theta_0) \geq \frac{\theta_0^3}{\pi^2}
	\end{equation}
	this is easy to see, because the third derivative at $0$ of the left hand side is greater than the one on the right, and both sides are increasing in $[0,\pi]$ and equality holds at both end-points.
	Then \cref{from_euler2,lem: euler bound1,lem: euler bound2,lem:euler_bound1+} imply
	\begin{equation} \label{lem: euler bound3}
		|v_2| \leq  \frac{2 \pi^2 \|y^s\|_{\infty}}{2 \theta_0} =   \frac{2 \pi^2 \|y^s\|_{\infty}}{|S|} \leq \frac{2 \pi^2 \|y\|_{\infty}}{|S|} .
	\end{equation}
    On the other hand, we may estimate by the generalized mean value theorem that
	\begin{align} \label{eq:mvt}
		\notag \fint_{S} y^s(x) \, x_2 \, dx  &= \fint_{-\theta_0}^{\theta_0} y^s(\theta)\sin \theta \, d \theta \\
		\notag &=  \fint_{-\theta_0}^{\theta_0} (y^s(\theta) - y^s(0)) \sin \theta \, d \theta\\
		\notag &= \fint_{-\theta_0}^{\theta_0} \left(\frac{y^s(\theta) - y^s(0)}{\theta}\right) \theta \, \sin \theta \, d \theta \\
		\notag &= (y^s)'(\tilde \theta)  \fint_{-\theta_0}^{\theta_0}  \theta \, \sin \theta \, d \theta\\ 
		&=  (y^s)'(\tilde \theta)  \left ( \frac{\sin(\theta_0)}{\theta_0} - \cos(\theta_0)\right ),
	\end{align}
	for some $\tilde \theta \in (- \theta_0, \theta_0)$. Then by $|\sin \theta_0 - \theta_0 \cos \theta_0| \leq \frac{\theta_0^3}{3}$ for  all $\theta_0 \in [0, \pi]$ we deduce by  \cref{from_euler2,lem: euler bound1,lem: euler bound2,eq:mvt} that
	\[
		|v_2| \leq  \pi^2  \|y^s\|_{C^1} \leq  \pi^2  \|y\|_{C^1}.
	\]
	This yields the second claim of the lemma, since $\frac{1}{\sqrt{m}} |\tilde v_S| \leq \sqrt{2} \max\{|v_1|, |v_2| \}$ and $|S| \leq \pi$.
	The first claim of the lemma follows from \cref{eq:wperp,from_euler3,lem: euler bound3} as
	\begin{align*}
    	\frac{1}{\sqrt{m}} |\tilde v_S\cdot w_i^\perp|\leq |v_1| + \underbrace{|\sin (\theta_0)|}_{\leq |S|/2}|v_2| \leq 4 \|y\|_{\infty} + \pi^2\|y\|_{\infty}.
	\end{align*}
\end{proof}
The previous lemma implies that if $y \in C^1$ then the minimizer $g$ of \cref{infimum-again} is Lipschitz continuous and we may bound the norms $\|g\|_{\infty}$ and $\|g\|_{C^1}$ with  $\|y\|_{\infty}$ and $\|y\|_{C^1}$ respectively. Let us write $g$ as in \cref{eq:g_repr} in which case the symmetric part is by \cref{lem:symasym}
\begin{equation}
    \label{g_for_unbound}
g^s(x) = \frac{1}{2\sqrt{m}} \sum_{i \in I}\sgn (\hat w_i \cdot x \geq 0) (v_i \cdot x).
\end{equation}
Next we  use the bounds for the minimizer $g$  from  \cref{lem:eulersecbound} to  bound  the vectors $v_i$  in \cref{g_for_unbound}.

\begin{lemma} \label{lem:ubound}
	Let $g$ and $y$ be as in \cref{lem:eulersecbound} and let us write $g^s$ as in \cref{g_for_unbound}. Let us further rewrite the vectors $v_i$ in \cref{g_for_unbound} as $v_i = \alpha_i \hat w_i +  u_i$, where $u_i$ is orthogonal to  $\hat w_i$ and  $\alpha_i \in \R$. Then
	\begin{align*}
		\frac{1}{\sqrt{m}}|u_i| \leq 6\pi^2\|y\|_{\infty}.
	\end{align*}
    and
    \begin{align*}
		\frac{1}{\sqrt{m}} |\alpha_i| \leq 6\pi^2  \|y \|_{C^1}.
	\end{align*}
\end{lemma}

Before the proof we note that the second inequality is a trivial consequence of \cref{lem:eulersecbound}. The difficulty is to obtain the first estimate, where we want the estimate to depend on the norm $\|y\|_{\infty}$ instead of $\|y \|_{C^1}$. 
\begin{proof}
    Let us fix a unit vector $\hat w_i$ in \cref{g_for_unbound} and by relabeling the nodes we call it $\hat w_1$. By rotating the coordinates we may assume that $\hat w_1 = e_2$. Let $S_1 \in \mathcal{S}$ be the sector of the form  $S_1 = \{x \in \S^1 :  0 \leq x_2 \leq \sin |S_{1}|, \, x_1 \geq 0 \}$ and let $S_2$ be the sector previous to $S_1$, in terms of the counterclockwise orientation. We write  $S_2 = \{x \in \S^1 : 0 \leq \hat w_2 \cdot x \leq \sin |S_{2}|, \, x_1 \geq 0 \} \subset \{ x_2 \leq 0\}$  for a unit vector $\hat w_2$. (We encourage the reader to draw a simple picture on the setting). Note that then $|S_2| \geq |\hat w_1 - \hat w_2|$. Let us write the function $g^s$ on sectors $S_1$ and $S_2$ as $g^s(x) = \frac{1}{\sqrt{m}} \tilde v_{S_i}(x)$, for $i = 1,2$,   as in \cref{tilde_v}. \cref{lem:eulersecbound} yields
	\begin{align} \label{eq:vs1bound}
		\frac{1}{\sqrt{m}}|\tilde v_{S_2}| \leq 3 \pi^2 \frac{ \|y\|_{\infty}}{|S_{2}|}.
	\end{align}
	
	By our construction we have $\sgn (\hat w_1 \cdot x) = 1$ for $x \in S_1$ and $\sgn (\hat w_1 \cdot x) =-1$ for $x \in S_2$. Therefore, by \cref{g_for_unbound} the vectors $\tilde v_{S_1}$ and $\tilde v_{S_2}$ differ only by  $\frac{1}{\sqrt{m}} v_1$, i.e., we have
	\begin{align} \label{eq:vs2bound}
		\frac{1}{\sqrt{m}} \tilde v_{S_1} - \frac{1}{\sqrt{m}} \tilde v_{S_2} = \frac{1}{\sqrt{m}} v_1 =   \frac{1}{\sqrt{m}}\alpha_1 \hat w_1 + \frac{1}{\sqrt{m}} u_1.
	\end{align}
	
	 Let us denote  $\hat u_1$ the unit vector such that $\hat u_1 \cdot u_1 = |u_1|$. (In fact since $\hat w_1 = e_2$ then $\hat u_1 = \pm e_1$.) In particular,  $\hat u_1$ is co-linear with $w_i^\perp$  and thus by \cref{lem:eulersecbound,eq:vs2bound} we have
	\begin{align} \label{eq:vs2dot1}
		 (4+\pi^2) \sqrt{m} \|y\|_{\infty}\geq |\tilde v_{S_1} \cdot \hat u_1 | = \big{|}\tilde v_{S_2} \cdot \hat u_1 + |u_1|\big{|}.
	\end{align}
	\cref{lem:eulersecbound} also yields 
		\begin{align} \label{eq:vs2dot2}
		|\tilde v_{S_2} \cdot \hat u_2 | \leq  (4+ \pi^2) \sqrt{m} \|y\|_{\infty},
		\end{align}
		where $\hat u_2$ is orthogonal to $\hat w_2$ and $\hat u_1 \cdot \hat u_2 >0$. Then, $|S_2| \geq |\hat w_2- \hat w_1| = |\hat u_2 - \hat u_1|$. Using this, \cref{eq:vs1bound,eq:vs2dot2} we obtain
    \begin{align*}
		|\tilde v_{S_2} \cdot \hat u_1| &\leq |\tilde v_{S_2} \cdot (\hat u_2 - \hat u_1)| + |\tilde v_{S_2} \cdot \hat u_2| \\
        &\leq \frac{3 \pi^2 \sqrt{m}  \|y\|_{\infty}}{|S_2|}|S_2| + (4 + \pi^2) \sqrt{m} \|y\|_{\infty}  \leq \sqrt{m} (4\pi^2 + 4)\|y\|_{\infty} .
    \end{align*}
	From the above and from \cref{eq:vs2dot1} we deduce that
	\begin{align*}
		\frac{1}{\sqrt{m}} |u_2| \leq (4+ \pi^2) \|y\|_{\infty} + (4\pi^2 + 4) \|y\|_{\infty} \leq
		6\pi^2 \|y\|_{\infty} .
	\end{align*}
    This yields  the first claim.

	For the second claim, we  use \cref{eq:vs2bound,lem:eulersecbound} (recall that $\hat w_1 \cdot u_1 = 0$) and get
    \[	
        \begin{split}
    		\frac{1}{\sqrt{m}} |\alpha_1| \leq \frac{1}{\sqrt{m}} |\tilde v_{S_2} | + \frac{1}{\sqrt{m}} |\tilde v_{S_1} |  \leq 6\pi^2  \|y\|_{C^1}.
        \end{split}
    \]
\end{proof}

We need yet two simple lemmas. In the first one we study the case when all the coefficients $a_i$ are positive.

\begin{lemma} \label{lemma easy1}
	Denote $W = (w_i)_{i=1}^m$. Then
	\[
		\fint_{\S^1} \left(  \frac{1}{\sqrt{m}}\sum_{i=1}^m \sigma(w_i \cdot x)\right)^2 \, dx \geq \frac{1}{4} \frac{|W|^2}{m}.
	\]
\end{lemma}

\begin{proof}
	\begin{align*}
		\fint_{\S^1} \left(  \frac{1}{\sqrt{m}}\sum_{i=1}^m \sigma(w_i \cdot x)\right)^2 dx \geq \frac{1}{m} \sum_{i=1}^m \fint_{\S^1} \sigma(w_i \cdot x)^2 dx = \frac{1}{m} \frac{1}{4} \sum_{i=1}^m |w_i|^2.
	\end{align*}
\end{proof}

The second simple result states that the minimum of the cost function in a ball is essentially independent of number of nodes, i.e, it only depends on the ratio between the number of nodes $m$ and the value $\underbar m = \min \{ |I_-|, |I_+|\}$.

\begin{lemma} \label{lemma easy2}
    Let $y \in L^2(\S^1)$ be as in the statement of \cref{thm2}. Assume that $m' < m$,  $ \underline m' = m'/2 \leq \underline m$ and denote $C(m) = \sqrt{m/\underbar m}$. Then for every $R>1$ we have
    \[
         \min_{\substack{f_W \in \Hma \\ |W| \leq C(m)R}}\| f_W - y \|_{2}^2 \leq  \min_{\substack{f_{W'} \in \mathcal{H}_{m',a'} \\ |W'| \leq R}}\| f_{W'} - y \|_{2}^2,
    \]
    where $\mathcal{H}_{m',a'}$ is defined in \cref{def:phiprim}.
\end{lemma}

\begin{proof} 
    We begin by recalling that by \cref{rem:reorder} the first $2 \underbar m$ coefficients are alternating $a_1 = -a_2 = \dots$ and, in particular,  all the coefficients up to $m'$ are alternating. Let us fix $R>1$ and let $W' \in \R^{2m'}$ be such that $|W'| \leq R$ and  
    \begin{align*}
        \| f_{W'} - y \|_{2}^2 = \min_{\substack{f_{W'} \in \mathcal{H}_{m',a'} \\ |W'| \leq R}}\| f_{W'} - y \|_{2}^2.
    \end{align*}
    We construct the vector $W_\lambda \in \R^{2m}$ as follows
    \begin{align*}
        W_\lambda = \lambda(\underbrace{W',W',\ldots,W'}_{\in \R^{2km'}},\underbrace{0}_{\in \R^{2m-2km'}}) \in \R^{2m},
    \end{align*}
    where $k=\lfloor{2\underbar{m}/m'}\rfloor$. From the definition of $W_\lambda$ we get
    \[
        f_{W_\lambda}(x) = \frac{k}{\sqrt{m}} \sum_{i=1}^{m'} a_i \sigma( \lambda w'_i \cdot x) 
       = k \lambda \frac{\sqrt{m'}}{\sqrt{m}}  f_{W'}(x).
    \]
    From this it is clear that if we choose $\lambda =  k^{-1} \frac{\sqrt{m}}{\sqrt{m'}}$ we obtain $f_{W_\lambda} = f_{W'}$. 
    Furthermore, by  $k \geq \underbar m / m'$  we have 
    \[
        |W|^2 = k \lambda^2 |W'|^2 \leq \frac{m}{\underbar m} |W'|^2 \leq  \frac{m}{\underbar m} R^2
    \]
    and the claim follows.
\end{proof}

We are now ready to prove \cref{thm2} which we restate here for the convenience of the reader.

\thmlocalization*

 \begin{remark} \label{rem:local-re}
 	The dependence on $m$ is necessary. This is easy to see e.g. by choosing $y(x) = -\sigma(x_1)$, with only one negative coefficient $a_1 = -1$ and the rest positive $a_i =1$ for $i = 2, \dots m$. Then it is clear that 
 	\[
 	\min_{\substack{f_W \in \Hma \\ |W| \leq R}}\| f_W - y \|_{2}^2
 	\]
 	is attained by choosing $w_1 = R e_1$ and $w_i = 0$ for  $i = 2, \dots m$. In this case 
 	\[
 		\min_{\substack{f_W \in \Hma \\ |W| \leq R}}\| f_W - y \|_{2}^2 \simeq  (1- \frac{R}{\sqrt{m}})_+^2 
 	\]
 	while $\inf_{f_W \in \Hma}\| f_W - y \|_{2}^2 = 0$.
 \end{remark}

\begin{proof}[\textbf{Proof of \cref{thm2}}]\phantom{.}
    \subsection*{Setup}
	Let $R$ be as in the statement of the theorem. 
	Furthermore, recall that $I_+ = \{ i \in I : a_i = 1 \}$ and $I_- = \{ i \in I : a_i = -1 \}$ and we assume that $\underbar m = \min \{ |I_+|, |I_-| \} = |I_-|$. We  also split $y$ into two parts $y = y_1 + y_2$, where $y_1 \in \Lal$, $y_2 \in \Lal^\perp$ and $\Lal$ is defined in \cref{L2al} as the space of functions for which the anti-symmetric part is linear. 
	
	One of the main difficulties in the proof is due to the fact that the problem behaves  differently depending on the value of $\underbar m$. When $\underbar m$ is large we may approximate $y_1$ well by \cref{thm3}, while on the other hand, the complexity of the function space $\Hma$ is large, which may lead to high cost of localization (too large weights). We need to find a good balance between these two factors. A further difficulty is due to the fact that \cref{thm3} is not useful in the case when $\underbar m$ small.  As such, we divide the  proof into two parts, for small and big values of $\underbar m$. We may clearly assume that $\underbar m \geq 1$

    \subsection*{Small $\underbar m$:}
    Assume that  $\underbar m^9 \leq R$.
	
	Let  $r \in \N_+$ be such that  $r \leq R^{1/3} \leq 2r$. We approximate $y$ with a smooth function  $y_r$ given by \cref{lem:approximation} and consider the minimization problem
	\begin{align} \label{eq:phiR}
		 \inf_{f_W \in \Hma} \|f_W - y_r\|_2^2  .
	\end{align}
	Note that \cref{lem:approximation} implies
	\begin{align}\label{eq:yyr}
	    \|y- y_r\|_{2}^2\leq  \frac{16 \|y\|_{BV}^2}{r}  \leq  \frac{32 \|y\|_{BV}^2}{R^{1/3}}  \leq 1,
	\end{align}
    for all  $R \geq R_0$ provided that $R_0 \geq (10 \|y\|_{BV})^6$.  Therefore from $\|y\|_{L^2} \leq 1$ we have  $\|y_r\|_{L^2}\leq 2$ and  
    \[
        \min_{\substack{f_W \in \Hma \\  |W|\leq \rho}} \|f_W - y_r\|_2^2\leq 4 
    \]
    for all $\rho >0$. Furthermore, let us fix $\rho >0 $ and  $W_\rho \in \R^{2m}$ such that $|W_\rho| \leq \rho$ and  
    \[
        \|f_{W_\rho} - y_r\|_2^2 = \min_{\substack{f_W \in \Hma \\  |W|\leq \rho}} \|f_W - y_r\|_2^2.
    \]
    Then by the previous estimates we have
    \[
    \begin{split} 
        \min_{\substack{f_W \in \Hma \\  |W|\leq \rho}} \|f_W - y\|_2^2  - &\min_{\substack{f_W \in \Hma \\  |W|\leq \rho}} \|f_W - y_r\|_2^2 \leq \|f_{W_\rho }- y\|_{2}^2  - \|f_{W_\rho }- y_r\|_{2}^2 \\
        &\leq ( \|f_{W_\rho }- y\|_{2}  + \|f_{W_\rho }- y_r\|_{2})\|y- y_r\|_{2}\\
        &\leq  (2\underbrace{\|f_{W_\rho }- y_r\|_{2}}_{\leq 2} + \underbrace{\|y- y_r\|_{2}}_{\leq 1})\|y- y_r\|_{2}\\
        &\leq 5\|y- y_r\|_{2}.
    \end{split}
    \]
    By a similar argument we obtain
    \begin{equation}
        \label{eq:forallrho-}
        	\Big{|} \min_{\substack{f_W \in \Hma \\  |W|\leq \rho}} \|f_W - y\|_2^2 - \min_{\substack{f_W \in \Hma \\  |W|\leq \rho}} \|f_W - y_r\|_2^2 \Big{|} \leq 5\|y- y_r\|_{2},
    \end{equation}
    for all $\rho \geq R_0$, where $R_0$ is as in the statement of the theorem.

	To continue, let $g \in \cHma$ be the minimum of \cref{eq:phiR}. Then, by \cref{thm4} we may write $g$ as
	\[
		g(x) =\frac{1}{\sqrt{m}} \sum_{i \in J }   \I\{ \hat w_i \cdot x \geq 0\} (u_i \cdot x) +  \frac{1}{\sqrt{m}}\sum_{i \in J }  \alpha_i \sigma( \hat w_i \cdot x) +\frac{1}{\sqrt{m}} \sum_{i \in K}   a_i \sigma( w_i \cdot x),
	\]
	where $|J| \leq \underbar{m}$, $|K| \leq m-2 (|J|)$,  $K \cap J = \emptyset$ and $u_i$ is orthogonal to $\hat w_i$. 
	
	\subsubsection*{Construction of a local approximation}
	The proof now goes as follows, we construct a function $f_{W_h} \in \Hma$, for which there is a $h_0$ such that $W_{h_0} \in R^{2m}$, $|W_{h_0}| \leq C(m)R$ and  $f_{W_{h_0}}$ approximates $g$ well. To prove this we first construct a family of functions $(f_{W_h})_h$ and prove that $|W_{h_0}| \leq C(m) R$ (for the right choice of $h_0$). Then we prove that $f_{W_{h_0}}$ approximates $g$ well. We choose the family of functions $(f_{W_h})_h$ for $h >0$ as follows
	\begin{multline} \label{eq:fdef}
	    f_{W_h}(x) =\frac{1}{\sqrt{m}} \sum_{i \in J} \left( \sigma \left( (\frac{1}{h} \hat w_i + u_i) \cdot  x\right)  - \sigma \left( (\frac{1}{h}  -  \alpha_i) \hat w_i \cdot  x\right)  \right) \\
		+\frac{1}{\sqrt{m}} \sum_{i \in K} a_i \sigma(w_i \cdot  x),
	\end{multline}
	where $W_h = (\frac{1}{h} \hat w_i + u_i,\ldots, \left (\frac{1}{h}-\alpha_i \right ) \hat w_i,\ldots, w_i,\ldots)$. It is clear that $f_{W_h} \to g$ as $h \to 0$. In order to bound $\|f_{W_{h_0}} - g\|_2$ it suffices to bound $\frac{d}{dh} \|f_{W_h} - g\|_2$ for $h \leq h_0$. 
	
	Note that \cref{lem:ubound} gives a bound for the components of $W_h$, which leads to a bound that depends on $m$.  In order to have an estimate independent of $m$ we need a more refined argument where we use the $L^2$ bound on $g$.
	
	To this end, let us split $g = g_+ + g_-$ where
	\[
        g_+(x) :=  \frac{1}{\sqrt{m}}  \sum_{K \cap I_+} \sigma(w_i \cdot x)
    \]
    and
    \[
        g_-(x) := \frac{1}{\sqrt{m}} \sum_{i \in J }   \I\{ \hat w_i \cdot x \geq 0\} (u_i \cdot x) +  \frac{1}{\sqrt{m}}\sum_{i \in J }  \alpha_i \sigma( \hat w_i \cdot x) -\frac{1}{\sqrt{m}} \sum_{i \in I_- \cap K}   \sigma( w_i \cdot x).
    \]
	We use $\|y_r\|_{L^2}\leq 2$ and the fact that $g$ is the minimizer of \eqref{eq:phiR}  to obtain $\|g\|_{L^2} \leq 4$. Then, 
	\begin{align*}
		\|g_+\|_2 \leq  \|g\|_2 + \|g_-\|_2 \leq 4 + \|g_-\|_2 .
	\end{align*}
	We are left to estimate the $L^2$ norm of $g_-$.  \cref{lem:approximation}  yields   $\|y_r\|_{C^1} \leq 5r \|y\|_{BV}  \leq 5 R^{1/3} \|y\|_{BV} $ and $\|y_r\|_{\infty} \leq \|y\|_{\infty} \leq \|y\|_{BV} $. Therefore  we have by  \cref{lem:ubound}
	that
	\begin{align}
		\label{eq:u_ibound}
		\frac{1}{\sqrt{m}}|u_i|\leq 6 \pi^2   \|y\|_{BV} 
	\end{align}
	and
	\begin{align}
		\label{eq:alphaibound}
		\frac{|\alpha_i|}{\sqrt{m}},\frac{|w_i|}{\sqrt{m}} \leq 30 \pi^2 R^{1/3}  \|y\|_{BV} 
	\end{align}
	for all $i \in J \cup K$. We  use \eqref{eq:u_ibound} and $|J| \leq \underbar m$ to estimate the first term in $g_-$ as
	\begin{align*}
	    \frac{1}{\sqrt{m}}	\left \| \sum_{i \in J }   \I\{ \hat w_i \cdot x \geq 0\} (u_i \cdot x) \right \|_2 \leq  6 \pi^2   \|y\|_{BV}  \underbar m
	\end{align*}
	and  \eqref{eq:alphaibound} and $|I_-| \leq \underbar m$ to estimate the two latter  terms 
	\begin{align*}
	    \frac{1}{\sqrt{m}}\left \| \sum_{i \in K \cap I_- } \sigma( w_i \cdot x) \right \|_2,\frac{1}{\sqrt{m}} \left \| \sum_{i \in J }  \alpha_i \sigma( \hat w_i \cdot x) \right \|_2 \leq 30 \pi^2  \|y\|_{BV} R^{1/3}\underbar{m}.
	\end{align*}
	Therefore by combining the two estimates we obtain 
	\begin{align*}
		\|g_-\|_2  \leq (6 \pi^2 + 60 \pi^2 R^{1/3})  \|y\|_{BV}    \underbar m
	\end{align*}
	which in turn yields
    \begin{align*}
		\|g_+\|_2 \leq 4 + \|g_-\|_2  \leq 4 +  (6 \pi^2 + 60 \pi^2  R^{1/3})  \|y\|_{BV}   \underbar m .
	\end{align*}
    
    We use the above inequality and  \cref{lemma easy1} to  obtain  
	\begin{align*}
	\left(\frac{1}{m} \frac{1}{4} \sum_{i \in K \cap I_+} |w_i|^2 \right)^{1/2} \leq 4 +  (6 \pi^2 + 60 \pi^2R^{1/3} )  \|y\|_{BV}    \underbar m
	\end{align*}
	which implies
	\begin{align} \label{eq:wi+bound}
		\left(\sum_{i \in K \cap I_+} |w_i|^2 \right)^{1/2} \leq 8 \sqrt{m} + 2 (6 \pi^2 + 60 \pi^2R^{1/3} )  \|y\|_{BV} \underbar m \sqrt{m}.
	\end{align}
	Recalling the definition of $W_h$ and using \cref{eq:u_ibound,eq:alphaibound,eq:wi+bound,eq:fdef} and the triangle inequality we have  the bound
	\begin{align*}
		\notag |W_{h}| &\leq \left(\frac{4\underbar{m}}{h^2} +  2 \left (  6^2\pi^4  +  30^2 \pi^4 R^{2/3}\right )\|y\|_{BV}^2\, m\underbar{m} + \sum_{i \in K \cap I_+} |w_i|^2 \right)^{1/2} \\
		\notag &\leq \frac{2 \sqrt{\underbar m}}{h} +  4 \left ( 6\pi^2  + 42 \pi^2 R^{1/3}\right )\|y\|_{BV} \,  \underbar m \sqrt{m} + 8 \sqrt{m}. 
	\end{align*}
	Recall that $1 \leq \underbar m^9 \leq R$. Then choosing $h_0 = R^{-1/2} (\underbar m \, m)^{-1/2}$  we obtain for $h = h_0$, $R \geq R_0 \geq \max \{(10\|y\|_{BV})^6, 4 \cdot 10^7\}$ that
	\begin{align} \label{eq:Wh0bound}
	    \notag |W_{h_0}| &\leq 2 \underbar m \sqrt{m} \sqrt{R} +  170 \pi^2 R^{1/3} \|y\|_{BV} \,  \underbar m \sqrt{m} + 8 \sqrt{m} \\
	    \notag &\leq (2 \sqrt{R} + 170 \pi^2 R^{1/3} \|y\|_{BV}+ 8) \left( \frac{m}{\underbar m}\right)^{1/2} \underbar m^{3/2} \\
	    \notag &\leq ( 3  +  17 \pi^2) C(m) \, \sqrt{R} {\underbar m}^{3/2} \\
	    \notag &\leq ( 3  +  17 \pi^2) C(m) \, R^{2/3} \\
	    &\leq \frac{1}{2} C(m) \, R.
	\end{align}
	
	Now that we have proved that our proposed approximation $f_{W_{h_0}}$ has a weight $W_{h_0} \in \R^{2m}$ which satisfies $|W_{h_0}| \leq C(m)R$. Next, let us prove that $f_{W_{h_0}}$ approximates $g$ well. To this end let us denote
	\[
		F_{h}(x) :=   \frac{1}{\sqrt{m}} \sum_{i \in J} \frac{\sigma( (\hat w_i + h u_i) \cdot  x)  - \sigma( \hat w_i \cdot  x)}{h},
	\]
	and
	\begin{align*}
        G (x) := \frac{1}{\sqrt{m}} \sum_{i \in J} \I\{\hat w_i \cdot x \geq 0\} (u_i \cdot x).
	\end{align*}
	Then by \cref{eq:fdef} and by recalling the form of $g$ we have $f_{W_h} - g = F_h - G$ and as such $F_h \to G$ as $h \to 0$. As mentioned previously in the proof, we wish to bound $\| F_{h} - G \|_2$ by estimating $\frac{d}{dh} \| F_{h} - G \|_2$ for $0 < h < h_0$ and then integrating up to $h_0$. By Cauchy-Schwarz inequality we get
	\begin{align} \label{eq:ddhbound}
	    \frac{d}{dh} \| F_{h} - G \|_2 \leq  \| \frac{d}{dh}F_{h}\|_2.
	\end{align}
	Therefore we may estimate
    \begin{align} \label{eq:ddhbound1}
        \left \|\frac{d}{dh} F_{h} \right \|_2 &= \frac{1}{\sqrt{m} \, h^2} I_1
    \end{align}
    where we bound $I_1$ as
    \begin{align} \label{eq:ddhbound1+}
        \notag I_1 &= \Big{\|} \sum_{i \in J} h \sigma'((\hat w_i + h u_i) \cdot x) u_i \cdot x - \sigma((\hat w_i+hu_i)\cdot x) + \sigma( \hat w_i \cdot  x) \Big{\|}_2\\
        &\leq  \sum_{i \in J} \Big{\|} h \sigma'((\hat w_i + h u_i) \cdot x) u_i \cdot x - \sigma((\hat w_i+hu_i)\cdot x) + \sigma( \hat w_i \cdot  x) \Big{\|}_2 \\
        \notag &=   h \sum_{i \in J}  \Big{\|} \fint_{0}^h   [\sigma'((\hat w_i + h u_i) \cdot x)   -  \sigma'((\hat w_i + \tau u_i) \cdot x)  ] (u_i \cdot x) d\tau\Big{\|}_2 \\
        \notag &\leq h \sum_{i \in J} \left( \fint_{0}^h |u_i|^2   \fint_{\S^1}(\sigma'((\hat w_i + h u_i) \cdot x)   -  \sigma'((\hat w_i + \tau u_i) \cdot x)^2 \, dx  d\tau \right)^{1/2}.
    \end{align}
	Recalling that $\sigma'(t) = \I\{t \geq 0\}$ we get from \cref{eq:ddhbound,eq:ddhbound1,eq:ddhbound1+} that
	\begin{align} \label{eq:ddhbound2}
		&\frac{d}{dh} \| F_{h} - G\|_2 \leq  \frac{1}{\sqrt{m} \,  h} \sum_{i \in J}  |u_i| \left( \fint_{0}^h  I_{2,i}(h,\tau) d\tau \right)^{1/2},
	\end{align}
	where
	\begin{align} \label{eq:ddhbound2+}
	    I_{2,i}(h,\tau) = \fint_{\S^1} \Big{|}\I\{(\hat w_i + h u_i) \cdot x \geq 0\}   - \I\{(\hat w_i + \tau u_i) \cdot x \geq 0\}\Big{|} \, dx.
	\end{align}
	From a geometric consideration it is clear that
	\begin{align} \label{eq:I2ibound}
	    I_{2,i}(h,\tau) \leq h|u_i|,
	\end{align}
	for all $\tau \in (0,h)$. Therefore we deduce from \cref{eq:ddhbound2,eq:ddhbound2+,eq:I2ibound,eq:u_ibound} that
	\begin{align} \label{eq:ddhbound3}
		\frac{d}{dh} \| F_{h} - G\|_2\leq  \frac{1}{\sqrt{m \, h} }  \sum_{i \in J}  |u_i|^{3/2}  \leq   15 \pi^3 \|y\|_{BV}^{3/2}  \underbar m \, m^{1/4} \frac{1}{\sqrt{h} }.  
	\end{align}
	We integrate \cref{eq:ddhbound3} over $(0, h_0)$, recalling that $\lim_{h \to 0} F_{h} = G$ and recalling our choice $h_0 = R^{-1/2} (\underbar m \, m)^{-1/2}$ we have for $\underbar m^9 \leq R$ and $R \geq R_0 \geq (10 \|y\|_{Bv})^6 $ that 
	\begin{align} \label{eq:bound approxim}
		\notag \| f_{W_{h_0}} - g \|_2 &=  \| F_{h} - G \|_2    \leq 30 \pi^3 \|y\|_{BV}^{3/2} \underbar m \, m^{1/4} \sqrt{h_0} \\
		\notag &\leq 30 \pi^3 \|y\|_{BV}^{3/2} \underbar{m}^{3/4} R^{-1/4} \\
		\notag &\leq 30 \pi^3 \|y\|_{BV}^{3/2} R^{-1/6} \\
		&\leq  94  \|y\|_{BV}^{1/2}.
	\end{align}
	To sum up, we have established \cref{eq:Wh0bound,eq:bound approxim}, which tells us that we have found a good local approximation $f_{W_{h_0}}$ of $g$ such that $\|W_{h_0}\| \leq C(m)R$.
	
	\subsubsection*{Putting everything together}
	\newcommand{\HmaR}{\mathcal{H}_{m,a}^{R_1}}
	Let us denote $\mathcal{H}_{m,a}^R  := \{f_W: f_W \in \Hma, \|W\| \leq R\}$ and $R_1 = C(m)R/2$. From \cref{eq:Wh0bound,eq:fdef,eq:bound approxim} we obtain that 
    \begin{align} \label{eq:localizationr}
        \notag \min_{\HmaR}\|f_W - y_r\|_2^2  - \inf_{\Hma}\|f_W - y_r\|_2^2 
        &\leq \|f_{W_{h_0}} - y_r \|_2^2  - \|g- y_r \|_2^2\\
        \notag &\leq ( \|f_{W_{h_0}} - y_r \|_2  + \|g- y_r \|_2) \, \|f_{W_{h_0}} - g \|_2\\
        \notag &\leq ( \|f_{W_{h_0}} - g \|_2+ 2\underbrace{\|g- y_r \|_2}_{\leq 2}) \, \|f_{W_{h_0}} - g \|_2\\
        \notag &\leq 30 \pi^3 \|y\|_{BV}^{3/2} ( 94   \|y\|_{BV}^{1/2}+4) \, R^{-1/6}\\
        &\leq 10^5 (\|y\|_{BV}^{2}+1)\, R^{-1/6}.
    \end{align}
    Finally we write 
    \begin{align} \label{eq:half done-}
	    \notag \min_{\HmaR} \|f_W - y\|_2^2  -\inf_{\Hma} \|f_W - y\|_2^2 =& \min_{\HmaR}\|f_W - y\|_2^2 - \min_{\HmaR}\|f_W - y_r\|_2^2 \\
	    \notag &+ \min_{\HmaR}\|f_W - y_r\|_2^2  - \inf_{\Hma}\|f_W - y_r\|_2^2 \\
	    \notag &+ \inf_{\Hma} \|f_W - y_r\|_2^2 - \inf_{\Hma} \|f_W - y\|_2^2 \\
	    =:& D_1 + D_2 + D_3
 	\end{align}
 	and thus by \cref{eq:forallrho-,eq:localizationr,eq:yyr} we have
 	\begin{align} \label{eq:half done+}
 	    D_1+D_2+D_3 \notag \leq& 10 \|y-y_r\|_2 +  10^5 (\|y\|_{BV}^{2}+1)\, R^{-1/6} \\
	    \notag \leq& \left (60\|y\|_{BV} + 10^5 (\|y\|_{BV}^{2}+1) \right )\, R^{-1/6} \\
	    \leq& 4 \cdot 10^4 (\|y\|_{BV}^{2}+1)\, R^{-1/9},
 	\end{align}
    for all $R \geq R_0$ (which implies $R^{1/18} > 2.6$). 
    Assembling \cref{eq:half done-,eq:half done+} we get
    \begin{align} \label{eq:half done}
        \min_{\HmaR} \|f_W - y\|_2^2  -\inf_{\Hma} \|f_W - y\|_2^2 \leq 4 \cdot 10^4 (\|y\|_{BV}^{2}+1)\, R^{-1/9},
    \end{align}
    for all $R \geq R_0$, which concludes the proof in the "small $\underbar m$" case, i.e. $\underbar m^9 \leq R$.
	
	\subsection*{Big $\underbar m$}
    Assume that $\underbar m^9 > R$. 
    
    Let us split $y$ into two parts, $y = y_1 + y_2$ where $y_1 \in \Lal$ and $y_2 \in \Lal^\perp$.
	Then by \cref{anti-symmetric} we have
    \begin{align} \label{eq:symasym_split}
        \min_{\substack{f_W \in \Hma \\  |W|\leq \rho}}\|f_W - y\|_2^2 &= \min_{\substack{f_W \in \Hma \\  |W|\leq \rho}}\|f_W - y_1\|_2^2 + \|y_2\|_2^2,
    \end{align}
    for all $\rho>0$. Now choose $m' \leq 2\underbar m$ as an even number, such that $R^{1/9} \geq m'/2 \geq R^{1/9}/2$, and let us consider the minimization problem in $\mathcal{H}_{m',a'}$ (see \cref{def:phiprim}). Since $\underbar m' \leq R^{1/9}$ and  $C(m') = \sqrt{2}$ we may use
    \cref{eq:half done} to obtain
    \begin{multline} \label{eq:restricted_localization1}
        \min_{\substack{f_{W'} \in \mathcal{H}_{m',a'} \\ |W'| \leq R}}\| f_{W'} - y \|_{2}^2  \leq  \inf_{f_{W'} \in \mathcal{H}_{m',a'} }\| f_{W'} - y \|_{2}^2 +  4 \cdot 10^4 (\|y\|_{BV}^{2}+1) R^{-1/9}.
    \end{multline}
    To bound the first term on the right hand side of \cref{eq:restricted_localization1} we use, $\underbar m' = m'/2 \geq R^{1/9}/2$, \cref{thm3,lem:lal} to get
    \begin{align} \label{eq:restricted_localization2}
        \inf_{f_{W'} \in \mathcal{H}_{m',a'} }  \|f_{W'} -y_1\|_2^2   \leq \frac{62 \|y_1\|_{BV}^2}{\underbar m'} \leq 2000 \|y\|_{BV}^2 R^{-1/9}.
    \end{align}
    Hence, from \cref{eq:symasym_split,eq:restricted_localization1,eq:restricted_localization2} we obtain the localization result for our restricted space $\mathcal{H}_{m',a'}$,
    \begin{align} \label{eq:restricted_localization3}
        \min_{\substack{f_{W'} \in \mathcal{H}_{m',a'} \\ |W'| \leq R}}\| f_{W'} - y \|_{2}^2  \leq  \|y_2\|_2^2 +   5 \cdot 10^4 (\|y\|_{BV}^{2}+1) R^{-1/9}.
    \end{align}
    To carry over the result to $\Hma$, we use \cref{lemma easy2} to get
    \begin{align} \label{eq:full_localization}
        \min_{\substack{f_{W} \in \Hma \\ |W| \leq C(m)R}}\| f_{W} - y \|_{2}^2 \leq  \min_{\substack{f_{W'} \in \mathcal{H}_{m',a'} \\ |W'| \leq R}}\| f_{W'} - y \|_{2}^2
    \end{align}
    and from the obvious fact that   
    \begin{align*}
        \inf_{f_{W} \in \Hma}\| f_{W} - y \|_{2}^2 \geq \|y_2\|_2^2.
    \end{align*}
    The estimate \cref{eq:full_localization} finishes the proof of \cref{thm2}.
    \end{proof}

\section*{Acknowledgments}
The first author was supported by the Swedish Research Council grant dnr: 2019-04098. The second author was supported by the Academy of Finland grant 314227.

\newpage

\bibliographystyle{acm}
\bibliography{references}

\end{document}